\documentclass[11pt,reqno]{amsart}
\usepackage{hyperref}

\allowdisplaybreaks[4]
\usepackage{verbatim}
\usepackage{amsmath}
\usepackage{amssymb}
\usepackage{mathrsfs}
\usepackage{amsthm}
\usepackage{bm}
\usepackage{bbm}
\usepackage{subfigure}
\usepackage{graphicx}
\usepackage{booktabs}
\usepackage{float}
\usepackage{cases}
\usepackage{geometry}
\usepackage{color}
\theoremstyle{plain}
\newtheorem{lemma}{Lemma}[section]
\newtheorem{thm}[lemma]{Theorem}

\newtheorem{prop}[lemma]{Proposition}
\newtheorem{assumption}{Assumption}

\begin{document}
\title[Superiority of stochastic symplectic methods via LIL]{
	Superiority of stochastic symplectic methods via the law of iterated logarithm
}
\author{Chuchu Chen, Xinyu Chen, Tonghe Dang, Jialin Hong}
\address{LSEC, ICMSEC,  Academy of Mathematics and Systems Science, Chinese Academy of Sciences, Beijing 100190, China,
	\and 
	School of Mathematical Sciences, University of Chinese Academy of Sciences, Beijing 100049, China}
\email{chenchuchu@lsec.cc.ac.cn; chenxinyu@amss.ac.cn; dangth@lsec.cc.ac.cn; hjl@lsec.cc.ac.cn}
\thanks{This work is funded by the National key R\&D Program of China under Grant (No. 2020YFA0713701), National Natural Science Foundation of China (No. 12031020), and by Youth Innovation Promotion Association CAS, China.}
\begin{abstract}
The superiority of stochastic symplectic methods over non-symplectic counterparts  has been verified by plenty of numerical experiments, especially in capturing  the asymptotic behaviour of the underlying solution process. 
How can one theoretically  explain this superiority?  
This paper gives an answer to this problem from the perspective of the law of iterated logarithm, taking the linear stochastic Hamiltonian system in Hilbert space as a test model. The main contribution is twofold. First, by fully utilizing the time-change theorem for martingales and the Borell--TIS inequality, we prove that the upper limit of the exact solution with a specific scaling function almost surely equals  some non-zero constant, thus confirming the validity of the law of iterated logarithm.
Second, we prove that stochastic symplectic fully discrete methods asymptotically preserve the law of iterated logarithm, but non-symplectic ones do not. 
This reveals the good ability of stochastic symplectic methods in characterizing the almost sure asymptotic growth of the utmost fluctuation  of  the underlying solution process.
Applications of our results to the linear stochastic oscillator and the linear stochastic Schr\"{o}dinger equation are also presented.
\end{abstract}
\maketitle
\section{Introduction}
The stochastic Hamiltonian system  (SHS) serves as a fundamental model in various physical and engineering sciences. 
One of the most relevant feactures of the SHS is that its phase flow preserves the stochastic symplectic structure pathwisely. 
Concerning the numerical approximation of the SHS, one may anticipate that numerical methods preserve the  symplecticity, leading to the pioneering works (see \cite{Milstein2002, Milstein20022}) of G. N. Milstein and co-authors on stochastic symplectic methods.
The construction and analysis of stochastic symplectic methods have further been  developed in recent decades; see monographs e.g. \cite{Hongsun2022,Hongwang2019,Milstein2021} and references therein.   There have been plenty of numerical experiments that 
demonstrate the superiority of stochastic symplectic methods over non-symplectic counterparts in the  long-term computation, especially in capturing  the asymptotic behaviour of the underlying solution process.  

The study on  the rigorous explanation of the superiority is a recent subject, and different perspectives are provided.  
One of the perspectives is based on the stochastic modified equation and the backward error analysis, providing insights into  the convergence properties of the truncated stochastic modified equation and the estimate of the Hamiltonian deviation for stochastic symplectic methods. 
For instance, for 
the stochastic Langevin equation, which is a special SHS, the backward error analysis of stochastic symplectic methods is performed by  the corresponding stochastic modified equation at the level of the stochastic differential equation (see \cite{Shardlow2006}) and at the level of the associated Kolmogorov equation (see \cite{Zygalakis2011, Kpoec20152}). 
For the SHSs with additive noise or with multiplicative noises for which the Hamiltonian $H_r(p,q), r\geq 1$ associated to the diffusion parts depend only on $p$ or only on $q$, authors in \cite{Wanghongsun2016} develop the approach of  constructing the stochastic modified equations of weak $k+k'$ order $(k'\geq 1)$ apart from the weak  $k$ order 
stochastic symplectic methods via their generating functions. 
The author in \cite{Anton2019} presents a backward error analysis for a stochastic symplectic method of  weak order one   by constructing the stochastic modified equation at the level of the associated Kolmogorov equation, and obtains an expansion of the weak error  associated with the numerical scheme.  
For the SHS driven by rough path, authors in \cite{chen2019stochastic}
construct a new type of stochastic modified equation for stochastic symplectic method, which is proved to have a Hamiltonian formulation, and obtain the pathwise convergence order of the truncated stochastic modified equation.  Authors in \cite{DAmbrosio2023} consider both the  It\^{o} SHS with separable Hamiltonian and additive noise and the  Stratonovich SHS, and present 
long-term estimates of the Hamiltonian deviation for stochastic symplectic methods by virtue of the corresponding stochastic modified equations and the backward error analysis.

Another perspective to theoretically explain  the superiority of stochastic symplectic methods is via the large deviation principle, which indicates  the good ability in approximating the exponential decay speed of rare event probabilities related to the exact solution.  
To be specific,
authors in \cite{chen2021asymptotically} and \cite{chen2023large}, 
taking the linear stochastic oscillator and the linear stochastic Schr\"{o}dinger equation 
as the test equations, respectively, prove that stochastic symplectic methods asymptotically preserve the large deviation principle of observables of the exact solution, but non-symplectic ones do not. From these results, atypically large deviations of observables from the average value are illustrated for the SHS and its stochastic symplectic methods.

In order to comprehend the good ability of stochastic symplectic methods in capturing the asymptotic behaviour of the underlying solution process, this paper presents a new perspective to
quantify the typical fluctuation of solution processes of the SHS and its stochastic symplectic methods as time goes to infinity.  
In classical probability theory, the fundamental probabilistic limit theorem for describing the maximum possible fluctuation of a stochastic process pathwisely over the long term is known as the law of iterated logarithm (LIL), providing the specific scaling function that characterizes the almost sure asymptotic growth of the process. 
One of the most important examples possessing the LIL is  the standard Brownian motion $B(t)$,  which satisfies $
\limsup_{t\rightarrow \infty} \frac{\left|B(t)\right|}{f(t)}= \sqrt{2}\; a.s.
$ with the scaling function $f(t)=\sqrt{t \log \log t}$; see e.g. \cite{schilling2014brownian}.  This result shows that  the utmost  fluctuation of Brownian motion asymptotically  grows as $\sqrt{2t \log \log t}$. 
With the aid of the LIL, we are going to investigate the following questions:
\begin{enumerate}
	\item[(\romannumeral1)] Does the SHS   obey the LIL with some scaling function $f$? Namely, is the upper limit of the exact solution with the scaling function $f$ a.s. equal to some non-zero constant $\gamma$? 
	\item[(\romannumeral2)] If so, is there a numerical method that also obeys the LIL with the same scaling function $f$? Namely, is the upper limit of the numerical solution with the scaling function $f$ a.s. equal to some non-zero constant $\gamma^\Delta$ which may depend on the mesh-size $\Delta$?  
	\item[(\romannumeral3)] Further, can the numerical method asymptotically preserve the LIL of the exact solution? Namely, does  $\gamma^\Delta$ converge to $\gamma$ as $\Delta$ takes limit?  
\end{enumerate}

In this paper, we take the linear SHS in the Hilbert space $\mathbb{U} \times \mathbb{U}$  
\begin{equation} \label{GeneralSHS}
	d\begin{pmatrix}
		X(t) \\
		Y(t)
	\end{pmatrix}=\begin{pmatrix}
		0 & B \\
		-B & 0
	\end{pmatrix}  \begin{pmatrix}
		X(t) \\
		Y(t)
	\end{pmatrix} d t+\begin{pmatrix}
		\alpha_1 \\
		\alpha_2
	\end{pmatrix} d W(t), \quad \begin{pmatrix}
		X(0) \\
		Y(0)
	\end{pmatrix} =
	\begin{pmatrix}
		X_0 \\
		Y_0
	\end{pmatrix}
\end{equation}
as a test model, and investigate the LILs for both the exact solution and its numerical methods.  Here,  $X, Y \in \mathbb{U}$ with $(\mathbb{U},\langle\cdot,\cdot \rangle_{\mathbb{R}})$ being a real-valued separable Hilbert space, $X_0,Y_0 \in \mathbb{U}$ are deterministic, $\{W(t)\}_{t\geq 0}$ is a $\mathbb{U}$-valued $Q$-Wiener process on a filtered probability space $(\Omega,\mathcal{F}, \{\mathcal{F}_t\}_{t\geq 0},\mathbb{P})$, $B$ is a self-adjoint linear operator. Precise assumptions on $W$ and $B$ are given in Section \ref{SectionGeneral}. 
To prove the LIL for the exact solution of \eqref{GeneralSHS}, the difficulty lies in that the stochastic convolution in the mild form of the exact solution is not a martingale, leading to the non-trivial analysis. To overcome this difficulty, the key is to extract a suitable martingale difference series  providing the predominant  contribution to the upper limit, 
where the time-change theorem is fully utilized to prove the remainder after the extraction converges to 0. Based on this analysis, we present the lower bound of the LIL.  
To further prove the upper bound of the LIL, we introduce an approach based on the Borell--TIS inequality, through which the delicate estimate of the small probability that the exact solution deviates from the scaling function is established.  
With the detailed analysis, we finally obtain the LIL for $\mathscr{X}(t)$  with the scaling function $f(t)=\sqrt{t \log\log t}$, namely, 
$$\limsup_{t \rightarrow \infty} \frac{\mathscr{X}(t)}{\sqrt{t \log \log t}}  = \sqrt{\alpha_1^2+\alpha_2^2} \sup_{j\in \mathbb{N}^+} \sqrt{\eta_j} \quad a.s.,$$
where $\mathscr{X}(t)\in\{\|X(t)\|_{\mathbb{R}}, \|Y(t)\|_{\mathbb{R}}, (\|X(t)\|^2_{\mathbb{R}} +\|Y(t)\|^2_{\mathbb{R}})^{\frac{1}{2}}\}$, and $\{\eta_j\}_{j\in\mathbb{N}^+}$ is the sequence of eigenvalues of $Q$. 

To further study the LIL for numerical methods, we apply the spectral Galerkin method in spatial direction, and a class of  one-step numerical methods in temporal direction to obtain a general class of fully discrete numerical methods $\{X^{M,\tau}(t_n),Y^{M,\tau}(t_n)\}_{n\in\mathbb{N}^+}$. Here, $M,\tau$ are the spectral projection dimension and time step-size, respectively.
The compact form of the numerical solution is formulated based on dimensionality reduction and iteration argument, with the explicit expression of coefficients being presented. The key in the proof of the LIL lies in the technical estimates for the discrete stochastic convolution of the numerical solution. We prove that the stochastic symplectic methods obey the LIL with the same scaling function $f$ as that of the exact solution case.  
Further, we present that as $\tau \rightarrow 0$ and $M\rightarrow\infty$, the limit for the LIL of stochastic symplectic methods coincide with the one for that of the exact solution, namely, 
$$
\lim_{M\rightarrow \infty} \lim _{\tau \rightarrow 0}\limsup_{n \rightarrow \infty} \frac{\mathscr{X}^{M,\tau}_{\mathrm{sym}}(t_n)}{\sqrt{ t_n \log \log t_n}}
= \sqrt{\alpha_1^2+\alpha_2^2} \sup_{j\in \mathbb{N}^+} \sqrt{ \eta_j} \quad a.s.,
$$ 
where $\mathscr{X}^{M,\tau}_{\mathrm{sym}}(t_n)$ denotes the numerical counterpart of $\mathscr{X}$ for stochastic symplectic methods. 
By contrast, we prove that the non-symplectic methods 
do not obey the LIL. 
Finally, we apply theoretical results to the linear stochastic oscillator and the linear stochastic Schr\"{o}dinger equation and obtain the LILs for both exact solutions and their stochastic symplectic methods. 
This paper is organized as follows: In Section \ref{SectionGeneral}, we prove the LIL for the exact solution of the linear SHS based on the time-change theorem for martingales and the Borell--TIS inequality. In Section \ref{SectionDiscrete}, we establish the LIL for stochastic symplectic methods of the linear SHS, and prove the asymptotic preservation of the LIL of the exact solution. Section \ref{SectionApplication} is devoted to applications of theoretical results to the linear stochastic oscillator and the linear stochastic Schr\"{o}dinger equation, respectively.

At the end of this section, we give some notations for the following content. We use $\log t$ to denote the natural logarithm $\log_e t$ and  use $\langle M \rangle (t)$ to denote the quadratic variation process of a martingale $M(t)$. Let $R=\mathcal{O}\left(h^p\right)$  denote $|R| \leq Ch^p$ for all sufficiently small $h$
and $f(h) \sim h^p$ claim that $f(h)$ and $h^p$ are equivalent infinitesimal. Denote by $\operatorname{Var}(\cdot)$ the variance for random variables. Let $\mathbf{i}$ be the imaginary unit. For a complex-valued number $x$, let $\Re x$  be its real part and $\Im x$ be its imaginary part. 
Throughout this paper, let $C$ denote an arbitrary constant which may vary from one line to another. 
\section{LIL for exact solution of linear SHS} \label{SectionGeneral}
In this section, we investigate the LIL for the linear SHS \eqref{GeneralSHS} based on the time-change theorem for martingales and the Borell--TIS inequality. 

We consider a densely defined, linear, self-adjoint and positive definite operator $B: \operatorname{dom} (B) \subset \mathbb{U} \rightarrow \mathbb{U}$ which is not necessarily bounded but with compact inverse. Suppose $B e_k=\lambda_k e_k$ for some non-decreasing sequence $\left\{\lambda_k\right\}_{k \in \mathbb{N}}$, where $\left\{e_k\right\}_{k \in \mathbb{N}}$ forms an orthonormal basis of 
$\left(\mathbb{U},\langle\cdot, \cdot\rangle_{\mathbb{R}}\right)$. Let $W(t)$ be a $\mathbb{U}$-valued $Q$-Wiener process  on a complete filtered probability space $(\Omega,\mathcal{F},\{\mathcal{F}_t\}_{t\geq 0},\mathbb{P})$, which can be represented as $W(t)=\sum_{k=1}^{+\infty} \sqrt{\eta_k} \beta_k(t) e_k$. Here, $\{\beta_{k}\}_{k=1}^\infty $ is a sequence of independent standard one-dimensional Brownian motions, 
and $Q$ is a non-negative symmetric operator on $\mathbb{U}$ with finite trace, whose eigenvalues and eigenvectors are, respectively, $\eta_k$ and $e_k$,  $k\in\mathbb{N}^+$.

The exact solution of \eqref{GeneralSHS} reads as
\begin{equation} \nonumber 
	\begin{aligned}
		X(t) 
		& =\cos(Bt)X_0+\sin(Bt)Y_0+\alpha_1 \int_0^t \cos ((t-s) B) d W(s)+\alpha_2 \int_0^t \sin ((t-s) B) d W(s), \\
		Y(t)& =-\sin(Bt)X_0+\cos(Bt)Y_0-\alpha_1 \int_0^t \sin ((t-s) B) d W(s)+\alpha_2 \int_0^t \cos ((t-s) B) d W(s) .
	\end{aligned}
\end{equation}

Below, we give the LIL result for the exact solution of the linear SHS. 
\begin{thm}\label{LILContinuous}
	For $\mathscr{X}(t) \in \{\|X(t)\|_{\mathbb{R}}, \|Y(t)\|_{\mathbb{R}}, (\|X(t)\|^2_{\mathbb{R}} +\|Y(t)\|^2_{\mathbb{R}})^{\frac{1}{2}}\}$, the LIL holds: 
	\begin{equation} \nonumber
		\begin{aligned}
			&\limsup_{t \rightarrow \infty} \frac{\mathscr{X}(t)}{\sqrt{ t \log \log t}}  = \sqrt{\alpha_1^2+\alpha_2^2} \sup_{j\in \mathbb{N}^+} \sqrt{\eta_j} \quad a.s.
		\end{aligned}
	\end{equation}
\end{thm}

To prove Theorem \ref{LILContinuous}, we introduce the following auxiliary process 
$$\tilde{X}(t) :=\cos(Bt)X_0+\sin(Bt)Y_0+\mathfrak{a}_1 \int_0^t \cos ((t-s) B) d W(s)+\mathfrak{a}_2 \int_0^t \sin ((t-s) B) d W(s) $$ with $\mathfrak{a}_1,\mathfrak{a}_2 \in \mathbb{C}$, and study the LIL for $\tilde{X}(t)$ in a complex-valued Hilbert space. To proceed, we give some notations.
Let $\mathbb{H}$ be the complex-valued Hilbert space corresponding to $\mathbb{U}$ with the complex inner product  $\langle \cdot, \cdot \rangle_{\mathbb{C}}$ and the real inner product $\langle\cdot, \cdot\rangle_{\mathbb{R}} = \Re \langle \cdot, \cdot \rangle_{\mathbb{C}} $.
We would like to mention that $\left\{e_k\right\}_{k \in \mathbb{N}}$ also forms an orthonormal basis of 
$\left(\mathbb{H},\langle\cdot, \cdot\rangle_{\mathbb{C}}\right)$.
\begin{prop}\label{LILContinuousGeneral}
	$\tilde{X}(t)$ obeys the following  LIL:
	\begin{equation} \nonumber
		\limsup_{t \rightarrow \infty} \frac{\|\tilde{X}\left(t\right)\|_{\mathbb{R}}}{\sqrt{ t \log \log t}}  =\sup_{\sum^\infty_{j=1}  (\rho_{1,j}^2+ \rho_{2,j}^2)=1}\sqrt{\phi(\rho)},
	\end{equation}
	where $\rho := \sum^\infty_{j=1} \left(\rho_{1,j} +\mathbf{i} \rho_{2,j}\right) e_j$ with $ \rho_{1,j},\rho_{2,j} \in \mathbb{R}$, and $$\phi(\rho) :=\sum^{\infty}_{j=1} ( (\Re\mathfrak{a}_1 \rho_{1,j} + \Im\mathfrak{a}_1 \rho_{2,j} )^2+(\Re\mathfrak{a}_2\rho_{1,j} + \Im\mathfrak{a}_2 \rho_{2,j} ) ^2)\eta_j.$$
\end{prop}
\begin{proof}[Proof of Theorem \ref{LILContinuous}]
	\textit{(\romannumeral1) Proof of the LIL for $\{\|X(t)\|_{\mathbb R}\}_{t\ge 0}$.}	Using Proposition \ref{LILContinuousGeneral} with $\mathfrak{a}_1=\alpha_1, \mathfrak{a}_2=\alpha_2$, we have
	\begin{equation} \nonumber
		\begin{aligned}
			\limsup_{t \rightarrow \infty} \frac{\left\|X\left(t\right)\right\|_{\mathbb{R}}}{\sqrt{ t \log \log t}}  = \sqrt{ \alpha^2_1 +\alpha^2_2}  \sup_{\sum^\infty_{j=1}  (\rho_{1,j}^2+ \rho_{2,j}^2)=1} (\sum^{\infty}_{j=1} \rho_{1,j}^2\eta_j)^{\frac{1}{2}}  \leq  \sqrt{\alpha_1^2+\alpha_2^2} \sup_{j\in \mathbb{N}^+} \sqrt{\eta_j}.
		\end{aligned}
	\end{equation}
	Since the  operator $Q$ has finite trace, there exist  $N_0\in\mathbb N^+$ and $j_0\in\{1,\ldots,N_0\}$ such that $\sup_{j\in \mathbb{N}^+} \eta_j = \sup_{j\in\{1,2,\ldots,N_0\}} \eta_j=\eta_{j_0}$. 
	Letting $$ \rho_{1,j} =
	\Big\{ \begin{aligned}
		1, \quad &j=j_0, \\
		0, \quad &j\neq j_0
	\end{aligned}
	$$
	and $\rho_{2,j}=0$ for all $j \in \mathbb{N}^+$, 
	we have 
	$$
	\sup_{\sum^\infty_{j=1}  (\rho_{1,j}^2+ \rho_{2,j}^2)=1}(\sum^{\infty}_{j=1} \rho_{1,j}^2\eta_j)^{\frac{1}{2}} \geq \eta_{j_0} = \sup_{j\in \mathbb{N}^+} \sqrt{\eta_j},
	$$
	which leads to
	\begin{equation} \nonumber
		\begin{aligned}
			\limsup_{t \rightarrow \infty} \frac{\left\|X\left(t\right)\right\|_{\mathbb{R}}}{\sqrt{ t \log \log t}}  = \sqrt{\alpha_1^2+\alpha_2^2} \sup_{j\in \mathbb{N}^+} \sqrt{\eta_j}.
		\end{aligned}
	\end{equation}
	
	\textit{(\romannumeral2) Proof of the LIL for $\{\|Y(t)\|_{\mathbb R}\}_{t\ge 0}$.} Similar to the proof of the LIL for $\{\|X(t)\|_{\mathbb R}\}_{t\ge 0}$, by taking  $\mathfrak{a}_1=\alpha_2, \mathfrak{a}_2=-\alpha_1$, we obtain the LIL result for $\{\|Y(t)\|_{\mathbb R}\}_{t\ge 0}.$
	
	\textit{(\romannumeral3) Proof of the LIL for $\{(\|X(t)\|^2_{\mathbb R}+\|Y(t)\|^2_{\mathbb R})^{\frac12}\}_{t\ge 0}$.}	Notice that $Z(t):= X(t)+\mathbf{i}Y(t) $ satisfies
	$$
	\begin{aligned}
		Z(t)&= \cos(Bt)(X_0+\mathbf{i}Y_0) +\sin(Bt) (Y_0 -\mathbf{i} X_0)\\
		&\quad +(\alpha_1+\mathbf{i}\alpha_2) \int_0^t \cos ((t-s) B) d W(s)+(\alpha_2-\mathbf{i}\alpha_1) \int_0^t \sin ((t-s) B) d W(s),
	\end{aligned}
	$$
	and  $\|Z(t)\|_{\mathbb{R}}=(\|X(t)\|^2_{\mathbb{R}} +\|Y(t)\|^2_{\mathbb{R}})^{\frac{1}{2}}$. According to Proposition \ref{LILContinuousGeneral} with $\mathfrak{a}_1=\alpha_1+\mathbf{i}\alpha_2, \mathfrak{a}_2=\alpha_2-\mathbf{i}\alpha_1$, we derive
	\begin{align*}
		\limsup_{t \rightarrow \infty} \frac{\|Z(t)\|_{\mathbb{R}}}{\sqrt{ t \log \log t}}
		&=   \sup_{\sum^\infty_{j=1}  (\rho_{1,j}^2+ \rho_{2,j}^2)=1}\Big(\sum^{\infty}_{j=1} \left( (\alpha_1\rho_{1,j} + \alpha_2 \rho_{2,j} )^2+(\alpha_2\rho_{1,j} -\alpha_1 \rho_{2,j} ) ^2\right)\eta_j \Big)^{\frac{1}{2}} \\
		&=   \sup_{\sum^\infty_{j=1}  (\rho_{1,j}^2+ \rho_{2,j}^2)=1} \sqrt{\alpha_1^2+\alpha_2^2}\Big(\sum^{\infty}_{j=1} (\rho^2_{1,j} + \rho^2_{2,j})   \eta_j \Big)^{\frac{1}{2}} 
		\leq \sqrt{\alpha_1^2+\alpha_2^2} \sup_{j\in \mathbb{N}^+} \sqrt{\eta_j}.
	\end{align*}
	By letting $$ \rho^2_{1,j} +\rho^2_{2,j}=
	\Big\{ \begin{aligned}
		1, \quad &j=j_0, \\
		0, \quad &j\neq j_0,
	\end{aligned}
	$$
	we conclude $$
	\begin{aligned}
		\limsup_{t \rightarrow \infty} \frac{(\|X(t)\|^2_{\mathbb{R}} +\|Y(t)\|^2_{\mathbb{R}})^{\frac{1}{2}}}{\sqrt{ t \log \log t}}  = \limsup_{t \rightarrow \infty} \frac{\|Z(t)\|_{\mathbb{R}}}{\sqrt{ t \log \log t}} =
		\sqrt{\alpha_1^2+\alpha_2^2} \sup_{j\in \mathbb{N}^+} \sqrt{\eta_j},
	\end{aligned}
	$$
	which finishes the proof.
\end{proof}

In the following proposition, we present
LILs for some martingales by means of the time-change theorem, which plays an important role in the proof of Proposition \ref{LILContinuousGeneral}. Define $M_1(t):= \int_{0}^{t} \cos \left(s B \right)d W(s)$, $M_2(t) := \int_{0}^{t} \sin \left(s B \right)d W(s)$. Then for all $j\in \mathbb{N}^+$, we have $$
\begin{aligned}
	M_{1,j}(t):= \langle M_1(t),e_j\rangle_{\mathbb{R}}=\langle \int_{0}^{t} \cos \left(s B \right)(\sum^\infty_{k=1} \sqrt{\eta_k} d\beta_{k} e_k ),e_j\rangle_{\mathbb{R}}= \sqrt{\eta_j} \int_{0}^{t} \cos \left(s \lambda_j \right) d\beta_{j},  \\
	M_{2,j}(t):= \langle M_2(t),e_j\rangle_{\mathbb{R}}=\langle \int_{0}^{t} \sin \left(s B \right)(\sum^\infty_{k=1} \sqrt{\eta_k} d\beta_{k} e_k ),e_j\rangle_{\mathbb{R}}= \sqrt{\eta_j} \int_{0}^{t} \sin \left(s \lambda_j \right) d\beta_{j}.
\end{aligned}
$$ 
\begin{prop} \label{RoughBoundGeneral}
	For $j \in \mathbb{N}^+, k=1,2$, $M_{k,j}(t)$ is a real-valued  martingale obeying the following LIL:
	\begin{equation} \nonumber 
		\begin{aligned}
			\limsup _{t \rightarrow \infty} \frac{\left|M_{k,j}(t)\right|}{\sqrt{ t \log \log t}} 	= \sqrt{\eta_j} \quad a.s.
		\end{aligned}  
	\end{equation}
\end{prop}
\begin{proof}
	Note that $M_{1,j}(t)$ is a real-valued martingale with 
	$$
	\begin{aligned}
		\langle M_{1,j}\rangle(t) &=\eta_j \int_{0}^{t} \cos ^{2}(s\lambda_j ) d s= \frac{t\eta_j}{2}+\frac{\eta_j}{4\lambda_j} \sin (2 t\lambda_j).
	\end{aligned}
	$$
	It follows from \cite[Theorem 3.4.6]{karatzas2012brownian} that $W_{1,j}(t):=M_{1,j} (T_{1,j}(t))$ is a one-dimensional Brownian motion  with
	$
	M_{1,j}(t)=W_{1,j}\left(\left\langle M_{1,j}\right\rangle(t)\right),
	$
	where $T_{1,j}(t) :=\inf \left\{s \geq 0 :\langle M_{1,j}\rangle(s)>t\right\}$.
	Since $\lim _{t \rightarrow \infty} \frac{2\langle M_{1,j} \rangle(t)}{t}=\eta_j$, we obtain
		\begin{align}
			\limsup _{t \rightarrow \infty} \frac{\left|M_{1,j}(t)\right|}{\sqrt{ t \log \log t}} &= \limsup _{t \rightarrow \infty} \frac{\left|W_{1,j}\left(\left\langle M_{1,j}\right\rangle(t)\right)\right|}{\sqrt{2 \left\langle M_{1,j}\right\rangle(t) \log \log \left(\left\langle M_{1,j}\right\rangle(t)\right)}} \frac{\sqrt{2 \left\langle M_{1,j}\right\rangle(t) \log \log \left(\left\langle M_{1,j}\right\rangle(t)\right)}}{\sqrt{ t \log \log t}} \notag \\ 
			&= \lim_{t \rightarrow \infty} \sqrt{\frac{ 2 \left\langle M_{1,j}\right\rangle(t) }{t}} 
			=\sqrt{\eta_j} \quad a.s., \label{LimsupOfM1jGeneral}
		\end{align}
	where the LIL for the Brownian motion $W_{1,j}(t)$ is used.
	Similarly, $M_{2,j}(t)$ is a martingale with
	$$
	\langle M_{2,j}\rangle(t) =\eta_j \int_{0}^{t} \sin ^{2}(s\lambda_j ) d s= \frac{t\eta_j}{2}-\frac{\eta_j}{4\lambda_j} \sin (2 t\lambda_j),
	$$
	and it can be proved that 
	$
	\limsup _{t \rightarrow \infty} \frac{\left|M_{2,j}(t)\right|}{\sqrt{ t \log \log t}}
	=\sqrt{\eta_j}\; a.s.
	$
	The proof is completed.
\end{proof}

With Proposition \ref{RoughBoundGeneral} in hand, 
we present the proof of Proposition \ref{LILContinuousGeneral} below.
\begin{proof}[Proof of Proposition \ref{LILContinuousGeneral}]
	For the sake of simplicity, we define $W_{\sin}(t) := \int^t_0 \sin ((t-s) B)dW(s)$, $W_{\cos}(t) := \int^t_0 \cos ((t-s) B)dW(s)$, and for $ j\in \mathbb{N}^+$, we define $W_{\sin,j}(t) :=\left\langle W_{\sin }(t), e_{j}\right\rangle_{\mathbb{R}}$, $W_{\cos,j}(t) :=\left\langle W_{\cos }(t), e_{j}\right\rangle_{\mathbb{R}}$.
	Using the Riesz representation theorem, we have
	\begin{equation} \nonumber
		\| \tilde{X}(t) \|^2_{\mathbb{R}} =\sup_{\rho\in\mathbb H,\|\rho\|_{\mathbb{R}}=1} |\langle \tilde{X}(t),\rho \rangle_{\mathbb{R}} |^2 .
	\end{equation}
	Hence, for all $\rho = \sum^\infty_{j=1} \left(\rho_{1,j} +\mathbf{i} \rho_{2,j}\right) e_j \in \mathbb{H}$  with $\rho_{1,j},\rho_{2,j} \in \mathbb{R}$ such that
	$
	\|\rho\|^2_{\mathbb{R}}= \sum^\infty_{j=1} (\rho_{1,j}^2 +\rho_{2,j}^2 ) = 1,
	$
	we obtain
		\begin{align*}
			&\quad \|\tilde{X}(t)\|^2_{\mathbb{R}} = \sup_{\rho\in\mathbb H,\|\rho\|_{\mathbb{R}}=1} |\langle \cos(Bt)X_0+\sin(Bt)Y_0,\rho \rangle_{\mathbb{R}} + \langle \mathfrak{a}_1 W_{\cos}(t)+\mathfrak{a}_2 W_{\sin}(t),\rho \rangle_{\mathbb{R}} |^2\\
			&= \sup_{\sum^\infty_{j=1} \left(\rho_{1,j}^2 +\rho_{2,j}^2 \right) = 1} \Big|\langle \cos(Bt)X_0+\sin(Bt)Y_0,\rho \rangle_{\mathbb{R}} + \sum^\infty_{j=1} \tilde{\mathfrak{a}}_{1,\rho,j} W_{\cos,j}(t) + \tilde{\mathfrak{a}}_{2,\rho,j} W_{\sin,j}(t)\Big|^2,
		\end{align*}
	where $\tilde{\mathfrak{a}}_{1,\rho,j}:=\Re\mathfrak{a}_1 \rho_{1,j}+ \Im\mathfrak{a}_1 \rho_{2,j}, \tilde{\mathfrak{a}}_{2,\rho,j}:=\Re\mathfrak{a}_2 \rho_{1,j}+ \Im\mathfrak{a}_2 \rho_{2,j}$.  
	For each sequence $\left\{(\rho_{1,j},\rho_{2,j})\right\}^\infty_{j=1}$ such that $\sum^\infty_{j=1} (\rho_{1,j}^2 +\rho_{2,j}^2 ) = 1$, let $			W_{\sin,\cos,\rho}(t):= \sum^\infty_{j=1} \tilde{\mathfrak{a}}_{1,\rho,j} W_{\cos,j}(t) + \tilde{\mathfrak{a}}_{2,\rho,j} W_{\sin,j}(t).
	$ 
	Then we divide the proof into two steps. \par
	\textit{Step 1: Lower bound of $\limsup_{t \rightarrow \infty} \frac{\|\tilde{X}(t)\|_{\mathbb{R}}}{\sqrt{ t \log \log t}}$}.
	\par
	Let $ m>2$, and 
	$t_{n}:=m^{n}$ for  $n\in\mathbb N$. We first give the estimate of $W_{\sin,\cos,\rho}(t_n)$, which can be decomposed as $$W_{\sin,\cos,\rho}\left(t_{n}\right)=A_{n,\rho}+B_{n,\rho}.$$
	Here,
	$$A_{n,\rho} := \sum^\infty_{j=1} \big\langle \int_{t_{n-1}}^{t_{n}} ( \tilde{\mathfrak{a}}_{1,\rho,j}  \left( \cos \left(t_{n}-s\right)B\right)+ \tilde{\mathfrak{a}}_{2,\rho,j}  \left( \sin \left(t_{n}-s\right)B\right) )d W(s),e_j \big\rangle_{\mathbb{R}} $$
	and
	$$B_{n,\rho}:= \sum^\infty_{j=1} \big\langle \int_{0}^{t_{n-1}} ( \tilde{\mathfrak{a}}_{1,\rho,j}  \left( \cos \left(t_{n}-s\right)B\right)+ \tilde{\mathfrak{a}}_{2,\rho,j}  \left( \sin \left(t_{n}-s\right)B\right) )d W(s),e_j \big\rangle_{\mathbb{R}} $$
	are independent Gaussian random variables for any given $n\in\mathbb{N}^+$.
	It is proved that $\{A_{n,\rho}\}_{n\in\mathbb{N}^+} $ is a martingale difference series with $\mathbb{E}\left[ A_{n,\rho}\right]=0$ and
	\begin{equation} \nonumber
		\begin{aligned}
			\operatorname{Var}\left(A_{n,\rho} \right) &= \sum^\infty_{j=1} \big\langle(\tilde{\mathfrak{a}}_{1,\rho,j} ^2 +\tilde{\mathfrak{a}}_{2,\rho,j} ^2) \frac{\Delta t_n Q}{2} e_j,e_j \big\rangle_{\mathbb{R}}+ \mathcal{J}(\Delta t_n,\rho) = \frac{\Delta t_n \phi(\rho)}{2} +\mathcal{J}(\Delta t_n,\rho),
		\end{aligned}
	\end{equation}
	where $\Delta t_{n} :=t_{n}-t_{n-1}$ 
	and  $$
	\begin{aligned}
		\mathcal{J}(\Delta t_n,\rho) &:=\sum^\infty_{j=1}\big((\tilde{\mathfrak{a}}_{1,\rho,j} ^2 -\tilde{\mathfrak{a}}_{2,\rho,j} ^2)\frac{\sin (2\lambda_j\Delta t_n)\eta_j}{4\lambda_j} 
		+\tilde{\mathfrak{a}}_{1,\rho,j} \tilde{\mathfrak{a}}_{2,\rho,j} (\frac{ \eta_j}{2\lambda_j} -\frac{ \cos(2\lambda_j \Delta t_n )\eta_j}{2\lambda_j})\big) .
	\end{aligned}
	$$
	Since $\{\lambda_j\}_{j=1}^\infty$ is non-decreasing, we know that
	\begin{align*}
		|\mathcal{J}(\Delta t_n,\rho)| 
		&\leq  \sum^\infty_{j=1} \big( |\tilde{\mathfrak{a}}_{1,\rho,j} ^2 -\tilde{\mathfrak{a}}_{2,\rho,j} ^2| \frac{\eta_j}{4\lambda_j} 
		+|\tilde{\mathfrak{a}}_{1,\rho,j} \tilde{\mathfrak{a}}_{2,\rho,j}| \frac{ \eta_j}{\lambda_j} \big) 
		\leq \sum^\infty_{j=1} \big(\frac{|\tilde{\mathfrak{a}}_{1,\rho,j} ^2-\tilde{\mathfrak{a}}_{2,\rho,j} ^2|}{4}+|\tilde{\mathfrak{a}}_{1,\rho,j} \tilde{\mathfrak{a}}_{2,\rho,j}|\big) \frac{ \eta_j}{\lambda_1}.
	\end{align*}
	Letting $\mathfrak{a}_{\max} := \max \{ |\Re \mathfrak{a}_1|, |\Im \mathfrak{a}_1|, |\Re \mathfrak{a}_2|, |\Im \mathfrak{a}_2| \}$, it holds that
	$$
	\begin{aligned}
		|\tilde{\mathfrak{a}}_{1,\rho,j} ^2-\tilde{\mathfrak{a}}_{2,\rho,j} ^2| \leq 2\mathfrak{a}^2_{\max} (\rho_{1,j}^2+\rho_{2,j}^2) \leq 2\mathfrak{a}^2_{\max},    \quad
		|\tilde{\mathfrak{a}}_{1,\rho,j} \tilde{\mathfrak{a}}_{2,\rho,j}|  \leq \mathfrak{a}^2_{\max} (|\rho_{1,j}|+|\rho_{2,j}|)^2 \leq 2\mathfrak{a}^2_{\max},
	\end{aligned}
	$$
	which leads to 
	\begin{equation} \label{givenJ0}
		\begin{aligned}
			|\mathcal{J}(\Delta t_n,\rho)|
			&\leq \frac{5\mathfrak{a}^2_{\max}}{2\lambda_{1}}  \operatorname{tr}(Q) 
			=: \mathcal{J}_0 < \infty.
		\end{aligned}
	\end{equation}
	Thus, we have
	$
		\operatorname{Var}\left(A_{n,\rho} \right) \leq \frac{\Delta t_n}{2} \phi(\rho) +\mathcal{J}_0.
	$
	Define 
	$C_{n,\rho}:=\big(\frac{\Delta t_{n}}{\beta}\phi(\rho) \log \log t_{n}\big)^{\frac{1}{2}} 
	$
	with $n\geq 2,\beta\in(1,2]$. Based on the fact that $A_{n,\rho}$ is  Gaussian, we derive 
		\begin{align*}
			\mathbb{P}\big\{A_{n,\rho}>C_{n,\rho}\big\} &=\mathbb{P}\big\{\frac{A_{n,\rho}}{\sqrt{\operatorname{Var}(A_{n,\rho})}}>\frac{C_{n,\rho}}{\sqrt{\operatorname{Var}(A_{n,\rho})}}\big\} \\
			&\geq \frac{1}{\sqrt{2 \pi}}\frac{1}{\frac{C_{n,\rho}}{\sqrt{\operatorname{Var}(A_{n,\rho})}}+\frac{\sqrt{\operatorname{Var}\left(A_{n,\rho}\right)}}{C_{n,\rho}}} \exp\big\{-\frac{C^2_{n,\rho}}{2 \operatorname{Var}(A_{n,\rho})}\big\}.
		\end{align*}
	For any $\epsilon_1 \in(0,\frac{3}{4}]$, there exists $N_0 := N_0(\epsilon_1,m,\rho) =\max\Big\{\big\lceil \frac{\log (\frac{4\mathcal{J}_0}{\phi \epsilon_1})}{\log m}\big\rceil, \big\lceil \frac{e^{2\beta}}{\log m} \big\rceil,  2 \Big\}$ such that for all $n > N_0$, it holds that
		\begin{align*}
			\frac{C_{n,\rho}}{\sqrt{\operatorname{Var}(A_{n,\rho})}} &=\sqrt{\big(\frac{1}{1+\frac{2\mathcal{J}}{\phi \Delta t_n}} \big) \frac{2}{\beta}\log\log t_n }
			\leq\sqrt{\big(\frac{1}{1-\frac{2\mathcal{J}_0}{\phi \Delta t_n}} \big) \frac{2}{\beta}\log\log t_n } \\
			&\leq \sqrt{\big(\frac{1}{1-\epsilon_1} \big) \frac{2}{\beta}\log\log t_n } 
			\leq 2 \sqrt{\frac{2}{\beta}\log\log t_n}
		\end{align*}
	and
	\begin{equation} \nonumber
		\begin{aligned}
			\frac{\sqrt{\operatorname{Var}(A_{n,\rho})}}{C_{n,\rho}} 
			\leq \sqrt{\frac{1+\frac{2\mathcal{J}}{\phi \Delta t_n}}{ \frac{2}{\beta}\log\log t_n}}
			\leq \frac{1+\frac{2\mathcal{J}_0}{\phi \Delta t_n}}{ \sqrt{\frac{2}{\beta}\log\log t_n}} \leq \frac{1}{2}+ \frac{\mathcal{J}_0}{\phi\Delta t_n} \leq 1 \leq \sqrt{\frac{2}{\beta}\log\log t_n}. 
		\end{aligned}
	\end{equation}
	Hence,
		\begin{align*}
			\mathbb{P}\big\{A_{n,\rho}>C_{n,\rho}\big\}
			&\geq \frac{1}{\sqrt{2 \pi}}\frac{1}{\frac{C_{n,\rho}}{\sqrt{\operatorname{Var}(A_{n,\rho})}}+\frac{\sqrt{\operatorname{Var}\left(A_{n,\rho}\right)}}{C_{n,\rho}}} \exp\Big\{-\frac{\log \log t_n }{\beta(1-\epsilon_1)}\Big\} \\
			&\geq \frac{1}{\sqrt{2 \pi}} \frac{1}{3\sqrt{\frac{2}{\beta} \left(\log n + \log \log m\right)}} \left(\log m \right)^{-\frac{1}{\beta(1-\epsilon_1)}} n^{-\frac{1}{\beta(1-\epsilon_1)}}.
		\end{align*}
	Taking $\epsilon_1 =\frac{1-1/\beta}{2} \leq \frac{1}{4}$, which implies $\beta\left(1-\epsilon_1\right) = \frac{\beta+1}{2}>1$,  we have $\sum_{n=N_0}^{\infty} n^{-\frac{1}{\beta\left(1-\epsilon_1\right)}}=\infty$. This  leads to
	\begin{equation}\nonumber 
		\sum_{n=2}^{\infty} \mathbb{P}\big\{A_{n,\rho}>C_{n,\rho}\big\}=\Big(\sum_{n=2}^{N_0}+\sum_{n=N_0+1}^{\infty}\Big) \mathbb{P}\big\{A_{n,\rho}>C_{n,\rho}\big\} = \infty .
	\end{equation}
	Noticing that  $\left\{A_{n,\rho}\right\}_{n=2}^{\infty}$ is a sequence of independent random variables, it follows from the Borel--Cantelli lemma that
	\begin{equation} \label{BCLemmaResultAnjGeneral}
		\mathbb{P}\Big\{  \big\{A_{n,\rho}>C_{n,\rho}\big\} \quad i.o.\Big\}=1 .
	\end{equation}
	By Proposition \ref{RoughBoundGeneral} and the Cauchy--Schwarz inequality, we obtain
	\begin{equation} \nonumber
		\begin{aligned}
			&\quad \limsup _{n \rightarrow \infty} \frac{\left|B_{n,\rho}\right|}{\sqrt{ t_{n-1} \log \log t_{n-1}}} \\
			&=\limsup _{n \rightarrow \infty}  \sum^{\infty}_{j=1} \frac{1}{\sqrt{ t_{n-1} \log \log t_{n-1}}}|\big\langle \big(\tilde{\mathfrak{a}}_{1,\rho,j} \cos(t_n B) + \tilde{\mathfrak{a}}_{2,\rho,j}  \sin (t_{n}B)\big) M_1(t_{n-1})  \\ 
			&\quad +\big( \tilde{\mathfrak{a}}_{1,\rho,j}  \sin (t_{n}B) -\tilde{\mathfrak{a}}_{2,\rho,j} \cos(t_n B)\big) M_2(t_{n-1}),e_j\big\rangle_{\mathbb{R}}| \\
			&\leq  \sum^{\infty}_{j=1} 2\sqrt{\eta_j} \left(\left|\tilde{\mathfrak{a}}_{1,\rho,j}  \right|+\left|\tilde{\mathfrak{a}}_{2,\rho,j}  \right|\right)  
			\leq 2 (\sum^{\infty}_{j=1} \left(\left|\tilde{\mathfrak{a}}_{1,\rho,j}  \right|+\left|\tilde{\mathfrak{a}}_{2,\rho,j}  \right|\right)^2 )^\frac{1}{2} (\sum^{\infty}_{j=1} \eta_j)^{\frac{1}{2}} 
			\leq 4\sqrt{2} \mathfrak{a}_{\max} (\operatorname{tr}(Q))^{\frac{1}{2}} \quad a.s.,
		\end{aligned}
	\end{equation}
	which implies
	\begin{equation}\label{LiminfOfBnGeneral}
		\begin{aligned}
			\liminf _{n \rightarrow \infty} \frac{B_{n,\rho}}{\sqrt{ t_{n-1} \log \log t_{n-1}}} 
			&\geq - 4\sqrt{2} \mathfrak{a}_{\max} (\operatorname{tr}(Q))^{\frac{1}{2}}  \quad a.s.  
		\end{aligned}
	\end{equation}
	According to the decomposition of $W_{\sin,\cos,\rho}(t)$, \eqref{BCLemmaResultAnjGeneral}, and \eqref{LiminfOfBnGeneral}, there exists a set $\tilde{\Omega} \subset \Omega$ with $\mathbb{P}(\tilde{\Omega})=1$ such that for any $\omega \in \tilde{\Omega}$ and $\varepsilon_1>0$, 
	\begin{equation} \label{WsincosInfiniteOftenGeneral}
		\begin{aligned} 
			W_{\sin,\cos,\rho}(t_n) 
			&> C_{n,\rho}+B_{n,\rho} \\
			&\geq \sqrt{ t_n \log \log t_n} 
			\big( \sqrt{\frac{1-\frac{1}{m}}{\beta} \phi(\rho)} 
			-(4\sqrt{2} \mathfrak{a}_{\max} (\operatorname{tr}(Q))^{\frac{1}{2}}+\varepsilon_1) \frac{\sqrt{ t_{n-1} \log \log t_{n-1}}}{\sqrt{ t_n \log \log t_n}}
			\big)
		\end{aligned}  
	\end{equation}
	holds for infinitely many $n> 1$.
	Noticing 
	$
	\lim _{n \rightarrow \infty} \frac{\sqrt{ t_{n-1} \log \log t_{n-1}}}{\sqrt{ t_{n} \log \log t_{n}}}=\sqrt{\frac{1}{m}} 
	$ 
	with $m>2$, and then letting $m\rightarrow \infty,\beta \rightarrow 1$, we deduce from \eqref{WsincosInfiniteOftenGeneral} that
	\begin{equation} \nonumber
		\begin{aligned}
			\limsup _{t \rightarrow \infty} \frac{\left|W_{\sin,\cos,\rho}(t)\right|}{\sqrt{ t \log \log t}} &\geq 	\limsup _{n \rightarrow \infty} \frac{ W_{\sin,\cos,\rho}(t_n) }{\sqrt{ t_n \log \log t_n}} 
			&\geq \sqrt{\phi(\rho)} \quad a.s. 
		\end{aligned}  
	\end{equation}
	Consequently,
		\begin{align*}
			\limsup _{t \rightarrow \infty} \frac{\|\tilde{X}(t)\|_{\mathbb{R}}}{\sqrt{ t \log \log t}} 
			&=  \limsup _{t \rightarrow \infty} \sup_{\sum^\infty_{j=1}  (\rho_{1,j}^2+ \rho_{2,j}^2)=1} \frac{\left|\langle \cos(Bt)X_0+\sin(Bt)Y_0,x \rangle_{\mathbb{R}}+  W_{\sin,\cos,\rho}(t) \right|}{\sqrt{ t \log \log t}} \\
			&\geq \limsup _{t \rightarrow \infty}  \frac{\left|W_{\sin,\cos,\rho}(t)\right|}{\sqrt{ t \log \log t}} 
			\geq  \sqrt{\phi(\rho)} \quad a.s. 
		\end{align*}  
	Since $\limsup _{t \rightarrow \infty} \frac{\|\tilde{X}(t)\|_{\mathbb{R}}}{\sqrt{t \log \log t}} $ is independent of $\rho$, taking supremum over all $\{(\rho_{1,j},\rho_{2,j})\}_{j=1}^\infty$ with $\sum^\infty_{j=1}\left(\rho_{1,j}^2 +\rho_{2,j}^2 \right) = 1$ on the right-hand side of the above inequality, we have the lower bound result
	\begin{equation} \label{XtsupGeneral}
		\begin{aligned}
			\limsup _{t \rightarrow \infty} \frac{\|\tilde{X}(t)\|_{\mathbb{R}}}{\sqrt{t \log \log t}} 
			&\geq \sup_{\sum^\infty_{j=1}  (\rho_{1,j}^2+ \rho_{2,j}^2)=1} \sqrt{\phi(\rho)} \quad a.s. 
		\end{aligned}  
	\end{equation}\par
	\textit{Step 2: Upper bound of $\limsup_{t \rightarrow \infty}\frac{\|\tilde{X}(t)\|_{\mathbb{R}}}{\sqrt{t \log \log t}}$}. \par
	Let $ m \in (1,2]$, and 
	$t_{n}:=m^{n}$ for $n\in\mathbb N$.  Define $f(t,\rho) := \sqrt{m^2 \phi(\rho) t \log \log t}.$
	According to the Burkholder--Davis--Gundy  inequality and the Cauchy--Schwarz inequality, for all $ n >N_1 := N_1(m) =  \max\{\lceil \frac{\log(\frac{1}{2\lambda_{1}})}{\log m}-1\rceil,0\}$, we have 
		\begin{align*}
			\mathbb{E}\big[\sup_{t\in [0,t_{n+1}]}  W_{\sin,\cos,\rho}(t) \big] 
			&\leq  \sum^{\infty}_{j=1} \big(\left| \tilde{\mathfrak{a}}_{1,\rho,j} \right|+\left|\tilde{\mathfrak{a}}_{2,\rho,j} \right|\big) \big( \mathbb{E}\big[\sup_{t\in [0,t_{n+1}]} \left| M_{1,j}(t)\right|\big] + \mathbb{E}\big[\sup_{t\in [0,t_{n+1}]}  \left|M_{2,j}(t)\right| \big] \big)\\
			&\leq \sum^{\infty}_{j=1} C \big(\left| \tilde{\mathfrak{a}}_{1,\rho,j} \right|+\left|\tilde{\mathfrak{a}}_{2,\rho,j} \right|\big) \big( \mathbb{E}\big((\langle M_{1,j}\rangle(t_{n+1}))^{\frac{1}{2}}\big) + \mathbb{E}\big((\langle M_{2,j}\rangle(t_{n+1}))^{\frac{1}{2}}\big) \big) \\
			&\leq \sum^{\infty}_{j=1} 2C \big(\left| \tilde{\mathfrak{a}}_{1,\rho,j} \right|+\left|\tilde{\mathfrak{a}}_{2,\rho,j} \right|\big) \big(\frac{t_{n+1} \eta_j}{2}+\frac{\eta_j}{4\lambda_j}\big)^{\frac{1}{2}} \\
			&\leq \sum^{\infty}_{j=1} 2C \big(\left| \tilde{\mathfrak{a}}_{1,\rho,j} \right|+\left|\tilde{\mathfrak{a}}_{2,\rho,j} \right|\big)  \big(t_{n+1} \eta_j\big)^{\frac{1}{2}} 
			\leq 4\sqrt{2}C \mathfrak{a}_{\max} \sqrt{\operatorname{tr}(Q) t_{n+1}}.
		\end{align*}
	Hence, there exists 
	$N_2:=N_2(m,\rho)=\max\Big\{\Big\lceil \frac{e^{\frac{C_1^2 \operatorname{tr} (Q)}{m\phi}}}{\log m} \Big \rceil , N_1\Big\}$ such that for all $n>N_2$,
	\begin{equation} \label{N2estimateGeneral}
		\mathbb{E}\big[\sup_{t\in [0,t_n]}  W_{\sin,\cos,\rho}(t) \big] \leq C_1 \sqrt{\operatorname{tr}(Q)}\sqrt{t_{n+1}} \leq  f(t_{n},\rho),
	\end{equation}
	where $C_1:=4\sqrt{2}C \mathfrak{a}_{\max} $.
	Therefore, for all $n>N_2$,  we obtain 
		\begin{align*}
			&\quad \big( f(t_{n},\rho) - \mathbb{E}\big[\sup_{t\in [0,t_{n+1}]}  W_{\sin,\cos,\rho}(t) \big] \big)^2 \\
			&= m^2 t_{n}\phi(\rho) \log \log t_{n} \Big(1+\frac{C_1^2\frac{\operatorname{tr}(Q)}{\phi(\rho) } -2C_1 \sqrt{m \frac{\operatorname{tr}(Q)}{\phi(\rho)}\log \log t_{n}}}{m\log \log t_{n} }\Big) .
		\end{align*}
	For any given $\epsilon_2>0$, there exists $N_3:= N_3(\epsilon_2,m,\rho)= \max \Big\{\Big\lceil \frac{e^{(\frac{4C_1^2\operatorname{tr}(Q)}{\epsilon_2^2 m\phi})}}{\log m} \Big\rceil ,N_2\Big\}$ such that for all $n>N_3$, it holds that
	$
		\frac{2C_1 \sqrt{m \frac{\operatorname{tr}(Q)}{\phi(\rho)}\log \log t_{n}}}{m\log \log t_{n} }  \leq \epsilon_2.
	$
	This implies that for all $n>N_3$,
	\begin{equation} \label{ft-EsupGeqGeneral}
		\begin{aligned}
			\big( f(t_{n},\rho) - \mathbb{E}\big[\sup_{t\in [0,t_{n+1}]} W_{\sin,\cos,\rho}(t) \big] \big)^2 
			&\geq \left(1-\epsilon_2 \right) m^2 \phi(\rho) t_{n} \log \log t_{n}.
		\end{aligned}
	\end{equation}
	Notice that
		\begin{align*}
			\sigma^2_{t_{n+1},\rho} &:= \sup_{t\in [0,t_{n+1}]} \mathbb{E} \big[W_{\sin,\cos,\rho}(t)\big]^2 
			= \sup_{t\in [0,t_{n+1}]} \Big\{ \sum^\infty_{j=1} \big\langle(\tilde{\mathfrak{a}}_{1,\rho,j}^2+\tilde{\mathfrak{a}}_{2,\rho,j}^2) \frac{tQ}{2} e_j,e_j \big\rangle_{\mathbb{R}}+\mathcal{J}(t,\rho)\Big\}\\
			&\leq \frac{t_{n+1}}{2}\phi(\rho) +\mathcal{J}_0,
		\end{align*}
	where $\mathcal{J}_0$ is given in \eqref{givenJ0}. 
	Then for  $\epsilon_2$, there exists $N_4 := N_4(\epsilon_2,m,\rho)= \max\Big\{\Big \lceil \frac{\log(\frac{2\mathcal{J}_0}{\epsilon_2 \phi})}{\log m}-1 \Big \rceil,0 \Big\}$ such that for all $n>N_4$,
	\begin{equation} \label{sigmatnboundGeneral}
		2 \sigma^2_{t_{n+1},\rho} \leq m\phi(\rho)  t_{n}(1+\epsilon_2).
	\end{equation}
	Applying the Borell--TIS inequality (see  \cite[Theorem 2.1.1]{adler2009random}) to $ W_{\sin,\cos,\rho}(t)$ with $t \in [0,t_{n+1}]$, and combining (\ref{ft-EsupGeqGeneral}) and  (\ref{sigmatnboundGeneral}), we have that for all $ n> N_5 := \max\{N_{3}, N_{4}\}$,
		\begin{align*}
			&\quad \; \mathbb{P}\Big\{\sup_{t\in [0,t_{n+1}]}  W_{\sin,\cos,\rho}(t) \geq f(t_{n},\rho) \Big\} \\
			&= \mathbb{P}\Big\{\sup_{t\in [0,t_{n+1}]}  W_{\sin,\cos,\rho}(t) - \mathbb{E}\big[\sup_{t\in [0,t_{n+1}]}  W_{\sin,\cos,\rho}(t) \big]  \geq f(t_{n},\rho)-\mathbb{E}\big[\sup_{t\in [0,t_{n+1}]}  W_{\sin,\cos,\rho}(t) \big] \Big\}  \\
			&\leq \exp\Big\{ - \frac{\big(f(t_{n},\rho)-\mathbb{E}\big[\sup_{t\in [0,t_{n+1}]}  W_{\sin,\cos,\rho}(t) \big]\big)^2}{2\sigma^2_{t_{n+1},\rho} }\Big\} \\
			&\leq \exp\Big\{ - \frac{\left(1-\epsilon_2 \right)m^2 \phi(\rho)t_{n} \log \log t_{n}  }{(1+\epsilon_2)m \phi(\rho) t_{n}}\Big\} 
			= \left(\log m\right)^{ -m\left(\frac{1-\epsilon_2}{1+\epsilon_2} \right)} n^{ -m\left(\frac{1-\epsilon_2}{1+\epsilon_2} \right)}.
		\end{align*}
	Taking $\epsilon_2<\frac{1-1/m}{1+1/m}$, which implies $ m\big(\frac{1-\epsilon_2}{1+\epsilon_2} \big)>1 $,  we have $\sum_{n=1}^{\infty} n^{-m\big(\frac{1-\epsilon_2}{1+\epsilon_2} \big)}<\infty$. This leads to
		\begin{align*}
			&\sum^\infty_{n=1} \mathbb{P} \Big\{ \sup_{t\in [0,t_{n+1}]} W_{\sin,\cos,\rho}(t) \geq f(t_{n},\rho) \Big\} \\
			=&  \Big(\sum^{N_5}_{n=1}+ \sum^\infty_{n=N_5+1} \Big) \mathbb{P} \Big\{ \sup_{t\in [0,t_{n+1}]}  W_{\sin,\cos,\rho}(t) \geq f(t_{n},\rho) \Big\} <\infty.
		\end{align*}
	Using the Borel--Cantelli lemma yields
	$
		\mathbb{P} \Big\{ \big\{\sup_{t\in [0,t_{n+1}]}  W_{\sin,\cos,\rho}(t) \geq f(t_{n},\rho)\big\} \quad  i.o. \Big\} =0,
	$
	which implies
		$
		\limsup_{n \rightarrow \infty} \frac{\sup_{t\in [0,t_{n+1}]}  W_{\sin,\cos,\rho}(t)}{\sqrt{m^2 \phi(\rho) t_{n} \log \log t_{n}}} \leq 1 \; a.s.
		$
	Hence, we arrive at
		\begin{align*}
			\limsup_{t \rightarrow \infty} \frac{ W_{\sin,\cos,\rho}(t)}{\sqrt{m^2 t \log\log t}} &\leq \limsup_{n \rightarrow \infty} \frac{\sup_{t\in [t_{n},t_{n+1}]}  W_{\sin,\cos,\rho}(t)}{\sqrt{m^2 t_{n} \log \log t_{n}}} \\
			&\leq \limsup_{n \rightarrow \infty} \frac{\sup_{t\in [0,t_{n+1}]}  W_{\sin,\cos,\rho}(t)}{\sqrt{m^2 t_{n} \log \log t_{n}}} 
			\leq \sqrt{\phi(\rho)}\quad a.s.
		\end{align*}
	Letting $m \rightarrow 1$, $W_{\sin,\cos,\rho}(t)$ obeys
	\begin{equation}
		\begin{aligned} \label{Wsincosupper1General}
			\limsup_{t \rightarrow \infty} \frac{ W_{\sin,\cos,\rho}(t)}{\sqrt{ t \log\log t}} \leq \sqrt{\phi(\rho)} \quad a.s.
		\end{aligned}
	\end{equation}
	Using the same method, we can also prove
	\begin{equation}
		\begin{aligned} \label{Wsincosupper2General}
			\limsup_{t \rightarrow \infty} \frac{ -W_{\sin,\cos,\rho}(t)}{\sqrt{ t \log\log t}} \leq \sqrt{\phi(\rho)}\quad a.s.
		\end{aligned}
	\end{equation}
	Combining \eqref{Wsincosupper1General} and \eqref{Wsincosupper2General}, it shows that
		\begin{align*}
			\limsup_{t \rightarrow \infty} \frac{ \left|W_{\sin,\cos,\rho}(t)\right|}{\sqrt{ t \log\log t}} \leq \sqrt{\phi(\rho)} \quad a.s.
		\end{align*}
	Due to the fact that 
	$
		\sup_{\|x\| = 1} \langle f,x\rangle =\langle f,\tilde{x} \rangle
	$
	with $\tilde{x} = \frac{f}{\|f\|}$, for any given $\omega \in \tilde{\Omega}$, we have
	\begin{equation} 
		\begin{aligned}
			\left| W_{\sin,\cos,\tilde{\rho}}(t,\omega)\right| =\sup_{\sum^\infty_{j=1}  (\rho_{1,j}^2+ \rho_{2,j}^2)=1} \left|W_{\sin,\cos,\rho}(t,\omega) \right|,
		\end{aligned}  
	\end{equation}
	where $\tilde{\rho}(\omega) := \frac{\mathfrak{a}_1 W_{\sin}(t,\omega)+\mathfrak{a}_2 W_{\cos}(t,\omega)}{\| \mathfrak{a}_1 W_{\sin}(t,\omega)+\mathfrak{a}_2 W_{\cos}(t,\omega)\|_{\mathbb{R}}}$.
	Therefore, 
		\begin{align*}
			\limsup _{t \rightarrow \infty} \frac{\big\|\tilde{X}(t,\omega)\big\|_{\mathbb{R}}}{\sqrt{t \log \log t}} 
			&=  \limsup _{t \rightarrow \infty}  \frac{\big|W_{\sin,\cos,\tilde{\rho}}(t,\omega)\big|}{\sqrt{ t \log \log t}} 
			\leq  \sqrt{\phi(\tilde{\rho}(\omega))}  
			\leq \sup_{\sum^\infty_{j=1}\left(\rho_{1,j}^2 +\rho_{2,j}^2 \right) = 1}\sqrt{\phi(\rho)}.
		\end{align*}  
	\par
	Combining \textit{Step 1} and \textit{Step 2}, 
	we finish the proof of Proposition \ref{LILContinuousGeneral}.
\end{proof}
\section{LIL for stochastic symplectic methods of linear SHS} \label{SectionDiscrete}
In this section, we study the LIL for fully discrete numerical methods of the linear SHS, whose spatial direction is based on the spectral Galerkin method and temporal direction is a class of one-step numerical methods. It is shown that the stochastic symplectic methods obey the LIL, but non-symplectic ones do not. Further, we show that stochastic symplectic methods asymptotically preserve the LIL of the exact solution case. 
\subsection{LIL for $\mathscr{X}^{M,\tau}_{\mathrm{sym}}(t_n)$}
For $M \in \mathbb{N}^+$, we define the M-dimensional subspace of $(\mathbb{U},\langle\cdot, \cdot\rangle_{\mathbb{R}})$ as $\mathbb{U}_M:=\operatorname{span}\left\{e_1, e_2, \ldots, e_M\right\}$ and the projection operator $P_M: \mathbb{U} \rightarrow \mathbb{U}_M$ as $P_M x=\sum_{k=1}^M\left\langle x, e_k\right\rangle_{\mathbb{R}} e_k$ for each $x \in \mathbb{U}$. Then $P_M : \mathbb{H} \rightarrow  \mathbb{H}_M$ is also a projection operator such that $P_M x=\sum_{k=1}^M\langle x, e_k\rangle_{\mathbb{C}} e_k$ for each $x \in \mathbb{H}$. 
Denote $B^M=P_M B,W^M=P_M W,X_0^M=P_M X_0,Y_0^M=P_M Y_0$. 
Using these notations, we obtain the following spectral Galerkin approximation:
\begin{equation}  \label{EquHamiltonianSystem}
	d\begin{pmatrix}
		X^{M}(t) \\
		Y^{M}(t)
	\end{pmatrix}=\begin{pmatrix}
		0 & B^M \\
		-B^M & 0
	\end{pmatrix}\begin{pmatrix}
		X^{M}(t) \\
		Y^{M}(t)
	\end{pmatrix} d t+\begin{pmatrix}
		\alpha_1 \\
		\alpha_2
	\end{pmatrix} d W^M(t), \quad \begin{pmatrix}
		X^{M}(0) \\
		Y^{M}(0)
	\end{pmatrix}=\begin{pmatrix}
		X^{M}_0 \\
		Y^{M}_0
	\end{pmatrix}.
\end{equation}
Letting $X^{k,M}(t) := \langle X^M(t), e_k\rangle_{\mathbb{R}}$, $Y^{k,M}(t) := \langle Y^M(t), e_k\rangle_{\mathbb{R}}$, \eqref{EquHamiltonianSystem} is equivalent to the following M subsystems
\begin{equation} \nonumber
	d\begin{pmatrix}
		X^{k,M}(t) \\
		Y^{k,M}(t)
	\end{pmatrix}= \begin{pmatrix}
		0 & \lambda_k \\
		-\lambda_k & 0
	\end{pmatrix}\begin{pmatrix}
		X^{k,M}(t) \\
		Y^{k,M}(t)
	\end{pmatrix} d t+\sqrt{\eta_k}\begin{pmatrix}
		\alpha_{1} \\
		\alpha_{2}
	\end{pmatrix} d \beta_k(t), \quad k=1,2, \dots, M.
\end{equation}
We further apply a class of one-step numerical methods in temporal direction to deriving the general full discretization $\{(X^{k,M, \tau}_{n}, Y^{k,M, \tau}_{n})\}_{n\in \mathbb{N}^+}$,
which satisfies the following equation
\begin{equation} \label{GeneralNumercialMethodForSchrodinger}
	\begin{pmatrix}
		X_{n+1}^{k,M, \tau} \\
		Y^{k,M, \tau}_{n+1}
	\end{pmatrix}=\begin{pmatrix}
		a_{11}\left(\lambda_k \tau\right) & a_{12}\left(\lambda_k \tau\right) \\
		a_{21}\left(\lambda_k  \tau\right) & a_{22}\left(\lambda_k  \tau\right)
	\end{pmatrix}\begin{pmatrix}
		X^{k,M, \tau}_{n} \\
		Y^{k,M, \tau}_{n}
	\end{pmatrix}+\sqrt{\eta_k}\begin{pmatrix}
		b_1\left(\lambda_k \tau\right) \\
		b_2\left(\lambda_k  \tau\right)
	\end{pmatrix} \delta \beta_{k, n}.
\end{equation}
Here, $\tau$ is the temporal step-size, $\delta \beta_{k, n} :=\beta_k\left(t_{n+1}\right)-\beta_k\left(t_n\right)$ with $t_n=n \tau, n=1,2, \ldots$, and $a_{i j}, b_i$ : $(0,+\infty) \rightarrow \mathbb{R}, i, j=1,2$ are  determined by a concrete numerical method. 
\par
Defining
$$
A(h):=\begin{pmatrix}
	a_{11}(h) & a_{12}(h) \\
	a_{21}(h) & a_{22}(h)
\end{pmatrix}, \quad b(h):=\begin{pmatrix}
	b_1(h) \\
	b_2(h)
\end{pmatrix}, \quad \forall \, h>0,
$$
we rewrite (\ref{GeneralNumercialMethodForSchrodinger}) into
\begin{equation} \label{RewriteNumercialMethodForSchrodinger}
	\begin{pmatrix}
		X_{n+1}^{k,M, \tau} \\
		Y^{k,M, \tau}_{n+1}
	\end{pmatrix}=A\left(\lambda_k  \tau\right)\begin{pmatrix}
		X^{k,M, \tau}_{n} \\
		Y^{k,M, \tau}_{n}
	\end{pmatrix}+\sqrt{\eta_k} b\left(\lambda_k \tau\right) \delta \beta_{k, n}, \quad n=0,1,2, \ldots
\end{equation}
We make the following assumption to obtain the compact form for the solution of \eqref{RewriteNumercialMethodForSchrodinger}. 
\begin{assumption} \label{AssumptionforGeneralFormula}
	There exists some $h_1> 0$ such that 
	$$4 \operatorname{det}\big(A(h)\big)-\big(\operatorname{tr}(A(h) )\big)^{2}>0, \quad \forall \, h<h_1.$$
\end{assumption}
Under Assumption \ref{AssumptionforGeneralFormula}, it follows from \cite{chen2023large} that for $k\in\{1,2,\ldots,M\}$ and sufficiently small $\tau$, $\big\{\big(X^{k,M, \tau}_{n}, Y^{k,M, \tau}_{n}\big)\big\}_{n \in \mathbb{N}^+}$ can be written into the compact form 
	\begin{align*}
		X_{n}^{k,M, \tau}&= -\operatorname{det}(A) \hat{\alpha}_{n-1}^{k} X^{k,M, \tau}_{0}+\hat{\alpha}_{n}^{k}\big(a_{11} X^{k,M, \tau}_{0}+a_{12} Y^{k,M, \tau}_{0}\big) \\
		&\quad +\sqrt{\eta_k} \sum_{j=0}^{n-1}\big(-\operatorname{det}(A) \hat{\alpha}_{n-2-j}^{k} b_{1}+\left(a_{11} b_{1}+a_{12} b_{2}\right) \hat{\alpha}_{n-1-j}^{k}\big) \delta \beta_{k, j}, \\
		Y^{k,M, \tau}_{n}&= a_{21} \hat{\alpha}_{n}^{k} X^{k,M, \tau}_{0}+\hat{\alpha}_{n+1}^{k} Y^{k,M, \tau}_{0}-a_{11} \hat{\alpha}_{n}^{k} Y^{k,M, \tau}_{0} \\
		&\quad +\sqrt{\eta_k} \sum_{j=0}^{n-1}\big(\left(a_{21} b_{1}-a_{11} b_{2}\right) \hat{\alpha}_{n-1-j}^{k}+b_{2} \hat{\alpha}_{n-j}^{k}\big) \delta \beta_{k, j},
	\end{align*}
where $\hat{\alpha}_n^k=\frac{(\operatorname{det}\left(A\right))^{\frac{n-1}{2}} \sin \left(n \theta_k\right)}{\sin \left(\theta_k\right)}$
with
$\theta_k \in(0, \pi)$ such that
$
\cos \left(\theta_k\right)=\frac{\operatorname{tr}\left(A\right)}{2 \sqrt{\operatorname{det}\left(A\right)}}$ and $\sin \left(\theta_k\right)=\frac{\sqrt{4 \operatorname{det}\left(A\right)-\left(\operatorname{tr}\left(A\right)\right)^2}}{2 \sqrt{\operatorname{det}\left(A\right)}}
$.
Here, $\operatorname{det}(A), a_{i j}, b_i \;(i, j=1,2)$ are computed at $\lambda_k \tau$.
For convenience, we always omit the argument $\lambda_k \tau$ in $\operatorname{det}(A), a_{i j}, b_i \; (i, j=1,2)$ when no confusion occurs.
\par
Let $X^{M,\tau}(t_n)=\sum^M_{k=1} X_n^{k,M,\tau} e_k \in \mathbb{U_M}, Y^{M,\tau}(t_n)=\sum^M_{k=1} Y_n^{k,M,\tau} e_k\in \mathbb{U_M}$. The stochastic symplectic fully discrete method obeys the LIL, which is stated in the following theorem.
\begin{thm} \label{LILfordiscrete}
	Let Assumption \ref{AssumptionforGeneralFormula} hold. If the numerical method \eqref{RewriteNumercialMethodForSchrodinger} is symplectic, then the LILs hold as follows:
		\begin{align}
			&	\limsup_{n \rightarrow \infty} \frac{\|X^{M,\tau}(t_n)\|_{\mathbb{R}}}{\sqrt{ t_{n} \log \log t_{n}}} = \max_{k\in\{1,2,\ldots,M\}} \sqrt{\xi_{k,1,\tau} \eta_k} \quad a.s., \label{XLIL}\\
			&	\limsup_{n \rightarrow \infty} \frac{\|Y^{M,\tau}(t_n)\|_{\mathbb{R}}}{\sqrt{ t_{n} \log \log t_{n}}} = \max_{k\in\{1,2,\ldots,M\}} \sqrt{\xi_{k,2,\tau} \eta_k} \quad a.s., \label{YLIL}\\
			&	\limsup_{n \rightarrow \infty} \label{XYLIL} \frac{(\|X^{M,\tau}(t_n)\|^2_{\mathbb{R}}+\|Y^{M,\tau}(t_n)\|^2_{\mathbb{R}})^{\frac{1}{2}}}{\sqrt{ t_{n} \log \log t_{n}}} = \sup_{\sum^M_{k=1}  (\rho_{1,k}^2+ \rho_{2,k}^2)=1} \sqrt{\phi^{M,\tau}(\rho)} \quad a.s. ,
		\end{align}  
	where 
		\begin{align*}
			\xi_{k,1,\tau} &:= \frac{b_{1}^{2}+\left(a_{11} b_{1}+a_{12} b_{2}\right)^{2}-2 b_{1}\left(a_{11} b_{1}+a_{12} b_{2}\right) \cos \left(\theta_{k}\right)}{ \sin ^{2}\left(\theta_{k}\right)}, \\
			\xi_{k,2,\tau} &:= \frac{b_{2}^{2}+\left(a_{21} b_{1}-a_{11} b_{2}\right)^{2}+2 b_{2}\left(a_{21} b_{1}-a_{11} b_{2}\right) \cos \left(\theta_{k}\right)}{ \sin ^{2}\left(\theta_{k}\right)}, 
		\end{align*}
		and		$
		\phi^{M,\tau}(\rho) :=\sum^{M}_{k=1} \big(\xi_{k,1,\tau}\rho^2_{1,k}+\xi_{k,2,\tau}\rho^2_{2,k}+\xi_{k,3,\tau}\rho_{1,k}\rho_{2,k}\big)\eta_k
		$ with $\rho:=\sum^M_{k=1}(\rho_{1,k}+\mathbf{i}\rho_{2,k})e_k$, 
		\begin{align*}
			\xi_{k,3,\tau} &:= \frac{-b_1b_2 \cos(2\theta_k) +[b_2(a_{11} b_{1}+a_{12} b_{2})-b_1\left(a_{21} b_{1}-a_{11} b_{2}\right)] \cos (\theta_k)}{\sin^2(\theta_k)} \\
			&\quad +\frac{(a_{11} b_{1}+a_{12} b_{2})\left(a_{21} b_{1}-a_{11} b_{2}\right)}{\sin^2(\theta_k)}.
		\end{align*}
\end{thm}

To prove Theorem \ref{LILfordiscrete}, we provide some preliminaries.
For $k\in \mathbb{N}^+$, define real-valued martingales as follows
\begin{align*}
	\tilde{M}^{\tau}_{1,k}(n) &:=\frac{1}{\sin (\theta_k)}\sum_{j=0}^{n-1} \big[ -b_1\cos ((j+1) \theta_k) +\left(a_{11} b_{1}+a_{12} b_{2}\right) \cos(j\theta_k) \big]\delta \beta_{k,j}, \\
	\tilde{M}^{\tau}_{2,k}(n) &:=\frac{1}{\sin  (\theta_k)}\sum_{j=0}^{n-1} \big[ -b_1\sin ((j+1) \theta_k) +\left(a_{11} b_{1}+a_{12} b_{2}\right) \sin(j\theta_k) \big]\delta \beta_{k,j}, \\
	\tilde{M}^{\tau}_{3,k}(n) &:=\frac{1}{\sin  (\theta_k)}\sum_{j=0}^{n-1} \big[ b_2\cos ((j-1) \theta_k) +\left(a_{21} b_{1}-a_{11} b_{2}\right) \cos(j\theta_k) \big]\delta \beta_{k,j}, \\
	\tilde{M}^{\tau}_{4,k}(n) &:=\frac{1}{\sin  (\theta_k)}\sum_{j=0}^{n-1} \big[ b_2\sin ((j-1) \theta_k) +\left(a_{21} b_{1}-a_{11} b_{2}\right) \sin(j\theta_k) \big]\delta \beta_{k,j}.
\end{align*}
Then it is proved in the following proposition that  $\tilde{M}^{\tau}_{j,k}(n)$, $j=1,\ldots,4$  obey the LILs.
\begin{prop} \label{DiscreteRoughBoundSch}
	For $k\in \mathbb{N}^+$, martingales $\tilde{M}^{\tau}_{j,k}(n)$, $j=1,\ldots, 4$  obey the following LILs:
		\begin{align*}
			\limsup _{n \rightarrow \infty} \frac{\big|\tilde{M}^{\tau}_{j,k}(n) \big|}{\sqrt{t_n \log \log t_n}} &= \sqrt{\xi_{k,1,\tau}} \quad a.s. \;\; (j=1,2), \\
			\limsup _{n \rightarrow \infty} \frac{\big|\tilde{M}^{\tau}_{j,k}(n) \big|}{\sqrt{t_n \log \log t_n}} &= \sqrt{\xi_{k,2,\tau}} \quad a.s. \;\; (j=3,4).
		\end{align*}
\end{prop}
\begin{proof}
	Using  
	\begin{equation} \label{sumOfsincos}
		\begin{aligned}
			\sum^{n-1}_{j=0} \sin(2j\theta) =\frac{\cos\theta-\cos\big((2n-1)\theta\big)}{2\sin\theta},\quad
			\sum^{n-1}_{j=0} \cos(2j\theta) =\frac{1}{2}+\frac{\sin\big((2n-1)\theta\big)}{2\sin\theta},
		\end{aligned}
	\end{equation}
	we derive the quadratic variation processes of the martingales
	$$
	\begin{aligned}
		\langle \tilde{M}^{\tau}_{1,k}\rangle(n) 
		=\frac{\xi_{k,1,\tau}}{2} t_n +\mathcal{K}_1(k,n,\tau), \quad \langle\tilde{M}^{\tau}_{2,k}\rangle(n)=\frac{\xi_{k,1,\tau}}{2} t_n -\mathcal{K}_1(k,n,\tau), \\
		\langle \tilde{M}^{\tau}_{3,k}\rangle(n) 
		=\frac{\xi_{k,2,\tau}}{2} t_n +\mathcal{K}_2(k,n,\tau),  \quad 
		\langle \tilde{M}^{\tau}_{4,k}\rangle(n) 
		=\frac{\xi_{k,2,\tau}}{2} t_n -\mathcal{K}_2(k,n,\tau),
	\end{aligned}
	$$
	where
	\begin{align*}
		&\quad \mathcal{K}_1(k,n,\tau)\\ &:=\frac{1}{4\sin^2 (\theta_k)}\Big(\frac{\sin((2n-1)\theta_k)}{\sin (\theta_k)}+1\Big)\Big(b_1^2\cos(2\theta_k)+\left(a_{11} b_{1}+a_{12} b_{2}\right)^2-2b_{1}\left(a_{11} b_{1}+a_{12} b_{2}\right)\cos(\theta_k)\Big)\tau\\
		&\quad +\frac{1}{\sin^2 (\theta_k)}\left[-b_1 \cos(\theta_k)+\left(a_{11} b_{1}+a_{12} b_{2}\right)\right]b_1\sin(\theta_k)\Big(\frac{\cos(\theta_k)-\cos((2n-1)\theta_k)}{2\sin(\theta_k)}\Big)\tau,		\\
		&\quad \mathcal{K}_2(k,n,\tau)\\ &:=\frac{1}{4\sin^2 (\theta_k)}\Big(\frac{\sin((2n-1)\theta_k)}{\sin (\theta_k)}+1\Big)\Big(b_2^2\cos(2\theta_{k})+\left(a_{21} b_{1}-a_{11} b_{2}\right)^2+2b_{2}\left(a_{21} b_{1}-a_{11} b_{2}\right)\cos(\theta_k)\Big)\tau\\
		&\quad +\frac{1}{\sin^2 (\theta_k)}\left(b_2\cos(\theta_k)+a_{21} b_{1}-a_{11} b_{2}\right) b_2\sin(\theta_k)\Big(\frac{\cos(\theta_k)-\cos((2n-1)\theta_k)}{2\sin(\theta_k)}\Big)\tau.		
	\end{align*}
	Since $\mathcal{K}_1(k,n,\tau),\mathcal{K}_2(k,n,\tau)$ are uniformly bounded with respect to $n$, similar to the proof of \eqref{LimsupOfM1jGeneral}, we finish the proof.
\end{proof}

Based on Proposition \ref{DiscreteRoughBoundSch}, we give the proof of Theorem \ref{LILfordiscrete}.
\begin{proof}[Proof of Theorem \ref{LILfordiscrete}]
	At first, we prove the LIL for $(\|X^{M,\tau}(t_n)\|^2_{\mathbb{R}}+\|Y^{M,\tau}(t_n)\|^2_{\mathbb{R}})^{\frac{1}{2}}$.
	Let $u^{M,\tau}(t_n):=X^{M,\tau}(t_n)+\mathbf{i}Y^{M,\tau}(t_n)$. 
	By the Riesz representation theorem, we have
	$
	\| u^{M,\tau}(t_n) \|^2_{\mathbb{R}} =\sup_{\{\rho\in \mathbb{H}_M, \|\rho\|_{\mathbb{R}}=1\}} |\langle u^{M,\tau}(t_n),\rho \rangle_{\mathbb{R}} |^2. 
	$
	Hence, for all $\rho = \sum^M_{k=1} \left(\rho_{1,k} +\mathbf{i} \rho_{2,k}\right) e_k \in \mathbb{H}_M$  with $\rho_{1,k},\rho_{2,k} \in \mathbb{R}$ and
	$
	\|\rho\|^2_{\mathbb{R}}= \sum^M_{k=1} (\rho_{1,k}^2 +\rho_{2,k}^2 ) = 1,
	$
	we obtain
		\begin{align*}
			\big\|u^{M,\tau}(t_n)\big\|^2_{\mathbb{R}} 
			&= \sup_{\sum^M_{k=1} \left(\rho_{1,k}^2 +\rho_{2,k}^2 \right) = 1} \Big| \sum^M_{k=1}\big(\langle X^{k,M, \tau}_n e_k,\rho_{1,k}e_k \rangle_{\mathbb{R}}+ \langle Y^{k,M, \tau}_n e_k,\rho_{2,k}e_k\rangle_{\mathbb{R}}\big) \Big|^2.
		\end{align*}
		For convenience, we denote $X^{k,M, \tau}_n = C^{k,M,\tau}_1(n) +\sqrt{\eta_k} G^{k,M,\tau}_1(n)$ and $ Y^{k,M, \tau}_n = C^{k,M,\tau}_2(n) +\sqrt{\eta_k} G^{k,M,\tau}_2(n)$ with
			\begin{align*}
				C^{k,M,\tau}_1(n)&:=  -\hat{\alpha}_{n-1}^{k} X^{k,M, \tau}_{0}+\hat{\alpha}_{n}^{k}\big(a_{11} X^{k,M, \tau}_{0}+a_{12} Y^{k,M, \tau}_{0}\big),\\
				C^{k,M,\tau}_2(n)&:=  a_{21}\hat{\alpha}_{n}^{k} X^{k,M, \tau}_{0}+\hat{\alpha}_{n+1}^{k}Y^{k,M, \tau}_{0}-a_{11}\hat{\alpha}_{n}^{k} Y^{k,M, \tau}_{0}, \\
				G^{k,M,\tau}_1(n)&:=  \sum_{j=0}^{n-1}\big(- \hat{\alpha}_{n-2-j}^{k} b_{1}+\left(a_{11} b_{1}+a_{12} b_{2}\right) \hat{\alpha}_{n-1-j}^{k}\big) \delta \beta_{k, j},   \\
				G^{k,M,\tau}_2(n)&:=  \sum_{j=0}^{n-1}\big(\hat{\alpha}_{n-j}^{k} b_{2}+\left(a_{21} b_{1}-a_{11} b_{2}\right) \hat{\alpha}_{n-1-j}^{k}\big) \delta \beta_{k, j}. 
			\end{align*}
		Then
			$
			\big\|u^{M,\tau}(t_n)\big\|^2_{\mathbb{R}} 
			= \sup_{\sum^M_{k=1}  (\rho_{1,k}^2+ \rho_{2,k}^2)=1} \Big| \sum^M_{k=1}\big( \rho_{1,k}C^{k,M,\tau}_1(n)+ \rho_{2,k}C^{k,M,\tau}_2(n) 
			\big)+G^{M,\tau}_{\rho}(n) \Big|^2,
			$
		where
		$G^{M,\tau}_{\rho}(n) := \sum^M_{k=1}\sqrt{\eta_k} \big(\rho_{1,k}G^{k,M,\tau}_1(n)+ \rho_{2,k} G^{k,M,\tau}_2(n)\big).$ 
		Now we divide the proof of \eqref{XYLIL} into two steps. \par
		\textit{Step 1: Lower bound of $\limsup_{n \rightarrow \infty} \frac{\left\|u^{M,\tau}\left(t_{n}\right)\right\|_{\mathbb{R}}}{\sqrt{ t_{n} \log \log t_{n}}}$}.
		
		Let $ m>2$. We decompose $G^{M,\tau}_{\rho}(m^n)$  as 
		$G^{M,\tau}_{\rho}(m^{n}) =A^{M,\tau}_{m^n,\rho}+B^{M,\tau}_{m^n,\rho},$
		where
		\begin{align*}
			A^{M,\tau}_{m^n,\rho} &:= \sum^M_{k=1}  \Big(\rho_{1,k}\sum_{j=m^{n-1}}^{m^{n}-1} \big(- \hat{\alpha}_{m^n-2-j}^{k} b_{1}+\left(a_{11} b_{1}+a_{12} b_{2}\right) \hat{\alpha}_{m^n-1-j}^{k}\big) \delta \beta_{k, j} \\
			&\quad  + \rho_{2,k} \sum_{j=m^{n-1}}^{m^n-1}\big(\hat{\alpha}_{m^n-j}^{k} b_{2}+\left(a_{21} b_{1}-a_{11} b_{2}\right) \hat{\alpha}_{m^n-1-j}^{k}\big) \delta \beta_{k, j}\Big)\sqrt{\eta_{k}}
		\end{align*}
		and
		\begin{align*}
			B^{M,\tau}_{m^n,\rho}&:= \sum^M_{k=1}  \Big(\rho_{1,k}\sum_{j=0}^{m^{n-1}-1} \big(- \hat{\alpha}_{m^n-2-j}^{k} b_{1}+\left(a_{11} b_{1}+a_{12} b_{2}\right) \hat{\alpha}_{m^n-1-j}^{k}\big) \delta \beta_{k, j} \\
			&\quad  + \rho_{2,k} \sum_{j=0}^{m^{n-1}-1}\big(\hat{\alpha}_{m^n-j}^{k} b_{2}+\left(a_{21} b_{1}-a_{11} b_{2}\right) \hat{\alpha}_{m^n-1-j}^{k}\big) \delta \beta_{k, j}\Big)\sqrt{\eta_{k}}
		\end{align*} 
		are independent Gaussian random variables for any given $n$. Note that  $\{A^{M,\tau}_{m^n,\rho}\}^{\infty}_{n=1}$ is a martingale difference series satisfying
		\begin{align*} 
				\operatorname{Var}\big(A^{M,\tau}_{m^n,\rho}\big) 
				&= \sum^M_{k=1} \sum_{j=m^{n-1}}^{m^{n}-1}\Big(\rho^2_{1,k}\big[- \hat{\alpha}_{m^n-2-j}^{k} b_{1}+(a_{11} b_{1}+a_{12} b_{2}) \hat{\alpha}_{m^n-1-j}^{k}\big]^{2}  \\
				&\quad +\rho^2_{2,k} \big[\hat{\alpha}_{m^n-j}^{k} b_{2}+\left(a_{21} b_{1}-a_{11} b_{2}\right) \hat{\alpha}_{m^n-1-j}^{k}\big]^{2} \\
				&\quad +\rho_{1,k}\rho_{2,k}\big[- \hat{\alpha}_{m^n-2-j}^{k} b_{1}+\left(a_{11} b_{1}+a_{12} b_{2}\right) \hat{\alpha}_{m^n-1-j}^{k}\big] \times\\
				&\quad\big[\hat{\alpha}_{m^n-j}^{k} b_{2}+\left(a_{21} b_{1}-a_{11} b_{2}\right) \hat{\alpha}_{m^n-1-j}^{k}\big] \Big) \eta_{k} \tau \\
				&=  \frac{\Delta t_{m^n}}{2} \phi^{M,\tau}(\rho) + \mathcal{J}^{M,\tau,\rho}(\Delta m^n),
		\end{align*}
		where $\Delta m^n := m^n-m^{n-1}, \Delta t_{m^n} := t_{m^n}-t_{m^{n-1}} = \Delta m^n \tau$, and
		{\small
			\begin{align*}
				\mathcal{J}^{M,\tau,\rho}(\Delta m^n) 
				&:= \sum^M_{k=1} \bigg[\Big( \rho^2_{1,k}\Big[-\frac{b_{1}^{2}+(a_{11} b_{1}+a_{12} b_{2})^{2}}{\sin^2 (\theta_k)} [1+\psi_{1,k}(\Delta m^n)] -b_1^2 \\ 
				&\quad +\frac{b_{1}\left(a_{11} b_{1}+a_{12} b_{2}\right)}{\sin^2 (\theta_k)} [2\cos (\theta_k) +\psi_{2,k}(\Delta m^n) ]
				+(a_{11} b_{1}+a_{12} b_{2})^2 \big(\frac{\sin ( \Delta m^n\theta_k)}{\sin (\theta_k)}\big)^2 \Big]\\
				&\quad +\rho^2_{2,k}\Big[-\frac{b_{2}^{2}+(a_{21} b_{1}-a_{11} b_{2})^{2}}{\sin^2 (\theta_k)} \big[\frac{1}{2}+\psi_{1,k}(\Delta m^n+1)\big]  +b_{2}^2 \frac{\sin^2 ( (\Delta m^n+1)\theta_k)}{\sin^2 (\theta_k)} \\
				&\quad  -\frac{b_{2}\left(a_{21} b_{1}-a_{11} b_{2}\right)}{\sin^2 (\theta_k)}\big[\cos(\theta_k)+\psi_{2,k}(\Delta m^n+1)\big] \Big] \\
				&\quad + \rho_{1,k}\rho_{2,k} \Big[\frac{b_1 b_2}{\sin^2 (\theta_{k})}\big[\frac{\cos(2\theta_k)}{2}+\psi_{1,k}(\Delta m^n+1)-\sin^2(\theta_{k})\big]\\
				&\quad+\frac{b_1\left(a_{21} b_{1}-a_{11} b_{2}\right)-b_2(a_{11} b_{1}+a_{12} b_{2})}{2\sin^2 (\theta_{k})}(2\cos(\theta_{k})+\psi_{2,k}(\Delta m^n)) \\
				&\quad+ b_2(a_{11} b_{1}+a_{12} b_{2})(\frac{\sin((\Delta m^n+1)\theta_{k})\sin(\Delta m^n\theta_{k})}{\sin^2 (\theta_{k})})\\
				&\quad -\frac{\left(a_{21} b_{1}-a_{11} b_{2}\right)(a_{11} b_{1}+a_{12} b_{2})}{\sin^2(\theta_{k})}\big[\frac{1}{2}+\psi_{1,k}(\Delta m^n +1)\big] \Big] \Big)\eta_k \tau \bigg].
		\end{align*} }
		Here, we denote
		$
		\psi_{1,k}(n) := \frac{\sin ((2n-3) \theta_k)-\sin (\theta_k)}{4 \sin (\theta_k)}, 
		\psi_{2,k}(n) := \frac{\sin (2 (n-1)\theta_k)-\sin  (2\theta_k)}{2 \sin (\theta_k)},
		$
		and use
		\begin{equation} \label{SumOfAlphaSymplectic}
			\begin{aligned}
				\sum_{j=0}^{n-2} (\hat{\alpha}^k_{j})^2 &=\frac{1}{\sin ^{2}(\theta_k)}\big(\frac{n-2}{2}-\frac{\sin ((2 n-3) \theta_k)-\sin (\theta_k)}{4 \sin (\theta_k)}\big), \\
				2 \sum_{j=1}^{n-1} \hat{\alpha}^k_{j} \hat{\alpha}^k_{j-1} &=\frac{1}{\sin ^{2}(\theta_k)}\big((n-2) \cos (\theta_k)-\frac{\sin (2 (n-1) \theta_k)-\sin (2 \theta_k)}{2 \sin (\theta_k)}\big)
			\end{aligned}
		\end{equation}
		for the stochastic symplectic method.
		Since $|\sin(\theta_{k})|\leq 1$ and $|\cos(\theta_{k})|\leq 1$, we can prove that $|\mathcal{J}^{M,\tau,\rho}(\Delta m^n)| \leq \mathcal{J}^{M,\tau}_0$ with some constant $\mathcal{J}^{M,\tau}_0$ independent of $n$.
		Define $C^{M,\tau}_{m^n,\rho}:=\big(\frac{\Delta t_{m^n}}{\beta} \phi^{M,\tau}(\rho) \log \log t_{m^n}\big)^{\frac{1}{2}}$.  Then for any $\epsilon_1\ \in(0,\frac{3}{4}]$, there exists a positive integer  $N_0:=N_0(\epsilon_1,m,\rho)=\max\Big\{\big\lceil \frac{\log (\frac{4\mathcal{J}^{M,\tau}_0}{\phi^{M,\tau} \epsilon_1})}{\log m}\big\rceil, \big\lceil \frac{e^{2\beta}}{\log m} \big\rceil, \big\lceil \frac{\log(e/\tau)}{\log m} \big\rceil\Big\}$ such that for all $n > N_0$, 
		\begin{align*}
			\mathbb{P}\big\{A^{M,\tau}_{m^n,\rho}>C^{M,\tau}_{m^n,\rho}\big\}
			&\geq \frac{1}{\sqrt{2 \pi}} \frac{1}{3\sqrt{\frac{2}{\beta} \left(\log n + \log \log m\right)}} \left(\log m \right)^{-\frac{1}{\beta(1-\epsilon_1)}} n^{-\frac{1}{\beta(1-\epsilon_1)}}.
		\end{align*}
	Taking $\epsilon_1 =\frac{1-1/\beta}{2} \leq \frac{1}{4}$ and then using the Borel--Cantelli lemma, we obtain 
	\begin{equation} \label{BCLemmaResultAnDiscrete}
		\mathbb{P}\Big\{\big\{ A^{M,\tau}_{m^n,\rho}>C^{M,\tau}_{m^n,\rho} \big\}\quad i.o.\Big\}=1.
	\end{equation}
	By Proposition \ref{DiscreteRoughBoundSch}, we derive 
		\begin{align*}
			&\quad \limsup _{n \rightarrow \infty} \frac{ |B^{M,\tau}_{m^n,\rho}|}{\sqrt{ t_{m^{n-1}} \log \log t_{m^{n-1}}}} \\
			&\leq \sum^{M}_{k=1} \Big(|\rho_{1,k}| \limsup _{n \rightarrow \infty}\Big( \frac{ \big| \sin \left(m^n \theta_k\right) \tilde{M}^{\tau}_{1,k}(m^{n-1})-\cos(m^n\theta_k) \tilde{M}^{\tau}_{2,k}(m^{n-1}) \big|}{\sqrt{ t_{m^{n-1}} \log \log t_{m^{n-1}}}}\Big) \\
			&\quad +|\rho_{2,k}|\limsup _{n \rightarrow \infty} \Big( \frac{\big| \sin \left(m^n \theta_k\right) \tilde{M}^{\tau}_{3,k}(m^{n-1})-\cos(m^n\theta_k) \tilde{M}^{\tau}_{4,k}(m^{n-1})\big| }{\sqrt{ t_{m^{n-1}}\log \log t_{m^{n-1}}}}\Big)  \Big) \sqrt{\eta_{k}} \\
			&\leq \sum^{M}_{k=1} 2\sqrt{\eta_{k}}\big(|\rho_{1,k}|\sqrt{\xi_{k,1,\tau}}+|\rho_{2,k}|\sqrt{\xi_{k,2,\tau}}\big) \\
			&\leq 2 \big(\sum^{M}_{k=1} (|\rho_{1,k}|\sqrt{\xi_{k,1,\tau}}+|\rho_{2,k}|\sqrt{\xi_{k,2,\tau}})^2 \big)^{\frac{1}{2}}\big(\sum^{M}_{k=1} \eta_{k}\big)^{\frac{1}{2}} 
			=: 2C_0^{M,\tau} \sqrt{\operatorname{tr}(Q)} \quad a.s.,
		\end{align*}
	where $C_0^{M,\tau}$ is some constant independent of $n$. 
	Therefore,
	\begin{equation}\label{LiminfOfBnDiscrete}
		\begin{aligned}
			\liminf _{n \rightarrow \infty} \frac{B^{M,\tau}_{m^n,\rho}}{\sqrt{t_{m^{n-1}} \log \log t_{m^{n-1}}}} 
			&\geq - 2C_0^{M,\tau}\sqrt{\operatorname{tr}(Q)} \quad a.s.  
		\end{aligned}
	\end{equation}
	Combining \eqref{BCLemmaResultAnDiscrete} and \eqref{LiminfOfBnDiscrete}, and using the same procedure as that for \eqref{XtsupGeneral} yield
	\begin{equation} \nonumber
		\begin{aligned}
			\limsup _{n \rightarrow \infty} \frac{\left\|u^{M,\tau}(t_n)\right\|_{\mathbb{R}}}{\sqrt{ t_n \log \log t_n}} 
			&\geq  \alpha \sup_{\sum^M_{k=1}  (\rho_{1,k}^2+ \rho_{2,k}^2)=1} \sqrt{\phi^{M,\tau}(\rho)} \quad a.s. 
		\end{aligned}  
	\end{equation}
	
	\textit{Step 2: Upper bound of $\limsup_{n \rightarrow \infty}\frac{\left\|u^{M,\tau}(t)\right\|_{\mathbb{R}}}{\sqrt{ t_n \log \log t_n}}$}. \par
	Let $ m \in (1,2]$. Define $f^{M,\tau}(t_{m^{n}},\rho) := \sqrt{m^2 \phi^{M,\tau}(\rho) t_{m^{n}} \log \log t_{m^{n}}}.$ 
	By the Burkholder--Davis--Gundy inequality, we obtain 
		\begin{align*}
			&\quad \mathbb{E}\Big[\sup_{ r \in \{1,2,\ldots,m^{n+1}\}} G^{M,\tau}_{\rho}(r) \Big] \\
			&\leq \sum^{M}_{k=1} \rho_{1,k}C\left( \mathbb{E}\Big[\big(\langle \tilde{M}^{\tau}_{1,k} \rangle(m^{n+1}) \big)^{\frac{1}{2}}\Big] + \mathbb{E}\Big[\big(\langle \tilde{M}^{\tau}_{2,k} \rangle(m^{n+1}) \big)^{\frac{1}{2}}\Big] \right)\sqrt{\eta_k} \\
			&\quad +\sum^{M}_{k=1}\rho_{2,k}C \left( \mathbb{E}\Big[\big(\langle \tilde{M}^{\tau}_{3,k} \rangle(m^{n+1}) \big)^{\frac{1}{2}}\Big] + \mathbb{E}\Big[\big(\langle \tilde{M}^{\tau}_{4,k} \rangle(m^{n+1}) \big)^{\frac{1}{2}}\Big] \right)\sqrt{\eta_k} \\
			&\leq  \sum^{M}_{k=1} 2\sqrt{\eta_{k}}C \Big(\rho_{1,k}  \big(\frac{\xi_{k,1,\tau}}{2}
			t_{m^{n+1}} +|\mathcal{K}_{1}(k,m^{n+1},\tau)| \big)^{\frac{1}{2}})+\rho_{2,k} \big(\frac{\xi_{k,2,\tau}}{2}
			t_{m^{n+1}} +|\mathcal{K}_{2}(k,m^{n+1},\tau)| \big)^{\frac{1}{2}}\Big).
		\end{align*}
	Notice that
	\begin{align*}
		|\mathcal{K}_1(k,n,\tau)| 
		&\leq \frac{1}{4\sin^2 (\theta_k)}\Big(\frac{1}{\sin (\theta_k)}+1\Big)\Big(b_1^2+\left(a_{11} b_{1}+a_{12} b_{2}\right)^2+2|b_{1}\left(a_{11} b_{1}+a_{12} b_{2}\right)|\Big)\tau\\
		&\quad +\frac{1}{\sin^2 (\theta_k)}|\big(-b_1 \cos(\theta_k)+\left(a_{11} b_{1}+a_{12} b_{2}\right)\big)b_1|\tau =: \mathcal{K}_{1,0}(k,\tau), \\
		|\mathcal{K}_2(k,n,\tau)| 
		&\leq\frac{1}{4\sin^2 (\theta_k)}\Big(\frac{1}{\sin (\theta_k)}+1\Big)\Big(b_2^2+\left(a_{21} b_{1}-a_{11} b_{2}\right)^2+2|b_{2}(a_{21} b_{1}-a_{11} b_{2})|\Big)\tau\\
		&\quad +\frac{1}{\sin^2 (\theta_k)}|\big(b_2\cos(\theta_k)+a_{21} b_{1}-a_{11} b_{2}\big) b_2|\tau =: \mathcal{K}_{2,0}(k,\tau).			
	\end{align*}
	Define $\mathcal{K}_{0}(k,\tau) := \max \{\mathcal{K}_{1,0}(k,\tau),\mathcal{K}_{2,0}(k,\tau)\}$.
	Then we have that for all
	$n>N_1:=N_1(m,\tau)=\max \big\{ 0, \big\lceil  \frac{\log (\frac{2\mathcal{K}_0}{\tau\xi_{k,j,\tau}})}{\log m}-1 \big\rceil: k=1,2,\ldots,M, j=1,2 \big\}$,
	\begin{align*}
		\mathbb{E}\Big[\sup_{ r \in \{1,2,\ldots,m^{n+1}\}} G^{M,\tau}_{\rho}(r) \Big] 
		&\leq  2C\sum^{M}_{k=1} \sqrt{\eta_{k}} \Big(\rho_{1,k} \sqrt{ \xi_{k,1,\tau} t_{m^{n+1}}}+\rho_{2,k}\sqrt{ \xi_{k,2,\tau} t_{m^{n+1}}} \Big)=:C_1\sqrt{t_{m^{n+1}}}.
	\end{align*}
	Similar to the proof of \eqref{N2estimateGeneral}, for all $n>N_2:=N_2(m,\rho,\tau)=\max \Big\{ \Big\lceil \frac{e^{\frac{C_1^2}{m \phi^{M,\tau}}}}{\log m} \Big\rceil, N_1 \Big\}$, we arrive at
	\begin{equation} \label{EGbound}
		\begin{aligned}
			\mathbb{E}\Big[\sup_{ r \in \{1,2,\ldots,m^{n+1}\}} G^{M,\tau}_{\rho}(r) \Big] \leq C_1 \sqrt{t_{m^{n+1}}} \leq f^{M,\tau}(t_{m^{n}},\rho).
		\end{aligned}
	\end{equation}
	Moreover, we have
	\begin{equation} \label{sigmabound}
		\begin{aligned}
			\tilde{\sigma}^2_{t_{m^{n+1}},M,\rho} :=\sup_{r \in \{1,2,\ldots,m^{n+1}\}}\mathbb{E}\Big[ ( G^{M,\tau}_{\rho}(r) )^2 \Big] \leq \frac{1}{2}\phi^{M,\tau}(\rho)t_{m^{n+1}} +\mathcal{J}^{M,\tau}_0.
		\end{aligned}
	\end{equation}
	Applying the Borell--TIS inequality to $G^{M,\tau}_{\rho}(r)$ with $r\in\{1,2,\ldots,m^n\}$ and combining \eqref{EGbound} and \eqref{sigmabound}, it yields that for any given $\epsilon_2>0$, when $n> N_3 := N_3(\epsilon_2,m,\rho) = \max\Big\{N_2,\Big\lceil \frac{\log \big(\frac{2\mathcal{J}_0}{\epsilon_2 \phi^{M,\tau}}\big)}{\log m}-1 \Big\rceil, \Big\lceil \frac{e^{\big(\frac{4C_1^2}{\epsilon_2^2 m\phi^{M,\tau}}\big)}}{\log m} \Big\rceil\Big\}$,
	\begin{align*}
		&\quad
		\mathbb{P}\Big\{\sup_{r\in\{1,2,\ldots,m^{n+1}\}} G^{M,\tau}_{\rho}(r)  \geq f^{M,\tau}(t_{m^{n}},\rho) \Big\} \\
		&\leq \exp\Big\{ - \frac{\big(f^{M,\tau}(t_{m^{n}},\rho)-\mathbb{E}\Big[\sup_{r\in \{1,2,\ldots,m^{n+1}\}} G^{M,\tau}_{\rho}(r)  \Big]\big)^2}{2\tilde{\sigma}^2_{t_{m^{n+1}},M,\rho}  }\Big\} \\
		&\leq \exp\Big\{ - \frac{ (1-\epsilon_2 )\phi^{M,\tau}(\rho) m^2 t_{m^{n}} \log \log t_{m^{n}}  }{ (1+\epsilon_2)\phi^{M,\tau}(\rho) m t_{m^{n}}}\Big\} 
		= (\log m)^{ -m(\frac{1-\epsilon_2}{1+\epsilon_2} )} \big(n+\frac{\log \tau}{\log m}\big)^{ -m(\frac{1-\epsilon_2}{1+\epsilon_2} )}.
	\end{align*}
	Using the Borel--Cantelli lemma, similar to the proof of \eqref{Wsincosupper1General}, we obtain
	$
	\limsup _{n \rightarrow \infty} \frac{G^{M,\tau}_{\rho}(n)  }{\sqrt{ t_n \log\log t_n}} \leq \sqrt{\phi^{M,\tau}(\rho)} \; a.s.
	$
	Similarly, we have
	$
	\limsup _{n \rightarrow \infty} \frac{-G^{M,\tau}_{\rho}(n)  }{\sqrt{ t_n \log\log t_n}} \leq \sqrt{\phi^{M,\tau}(\rho)} \; a.s.
	$
	This leads to
	$$
	\limsup _{n \rightarrow \infty} \frac{\big|G^{M,\tau}_{\rho}(n)  \big|}{\sqrt{ t_n \log \log t_n}} 
	\leq \sqrt{\phi^{M,\tau}(\rho)} \; a.s. 
	$$
	Therefore, for a.s. $\omega \in \Omega$, by taking $\rho_0(\omega)$ such that $|G^{M,\tau}_{\rho_0}(n,\omega)| = 	\sup\limits_{\sum^M_{k=1}  (\rho_{1,k}^2+ \rho_{2,k}^2)=1} |G^{M,\tau}_{\rho}(n,\omega) |,$
	we have
	\begin{align*}
		&\quad \limsup _{n \rightarrow \infty} \frac{(\left\|X^{M,\tau}(t_n)\right\|^2_{\mathbb{R}}+\left\|Y^{M,\tau}(t_n)\right\|^2_{\mathbb{R}})^\frac{1}{2}}{\sqrt{ t_n \log \log t_n}} \\
		&= \limsup _{n \rightarrow \infty} \frac{\left\|u^{M,\tau}(t_n)\right\|_{\mathbb{R}}}{\sqrt{ t_n \log \log t_n}} 
		= \limsup _{n \rightarrow \infty} \frac{|G^{M,\tau}_{\rho_0}(n,\omega)|}{\sqrt{ t_n \log \log t_n}}
		\leq   \sup_{\sum^M_{k=1}  (\rho_{1,k}^2+ \rho_{2,k}^2)=1} \sqrt{\phi^{M,\tau}(\rho)} \quad a.s. 
	\end{align*}
	Combining \textit{Step 1} and \textit{Step 2},  
	we finish the proof of \eqref{XYLIL}. 
	
	Now we give proofs of \eqref{XLIL} and \eqref{YLIL}. Define $G^{M,\tau,1}_{\rho} (n) :=\sum^M_{k=1} \sqrt{\eta_k}\rho_{1,k}G^{k,M,\tau}_1(n)$ and  $G^{M,\tau,2}_{\rho} (n) :=\sum^M_{k=1} \sqrt{\eta_k}\rho_{2,k}G^{k,M,\tau}_2(n)$. Using the above result on $G^{M,\tau}_{\rho} (n)$ with $\{\rho_{2,k}=0: k=1,2,\ldots,M\}$ and $\{\rho_{1,k}=0: k=1,2,\ldots,M\}$,  respectively, and noticing the relation holds as $ \sup_{\sum^M_{k=1} \rho^2_{j,k}=1} \sqrt{\sum^M_{k=1} \rho^2_{j,k} \xi_{k,j,\tau}\eta_{k}} = \max_{k\in\{1,2,\ldots,M\}} \sqrt{\xi_{k,j,\tau}\eta_{k}}$, $j=1,2$, we finish the proof.
\end{proof}
\subsection{Asymptotic preservation for LILs} 
Recalling that in Theorem \ref{LILContinuousGeneral} and Theorem \ref{LILfordiscrete}, we acquire the LILs for the exact solution of the linear SHS and the numerical solution of the stochastic symplectic method. 
This subsection is devoted to showing the asymptotic preservation of the LIL by the  stochastic symplectic method. 
As a comparison, we also give the result on non-symplectic methods, which fail to obey the LIL.
To this end, we need the following assumption on the convergence of numerical methods. 
\begin{assumption} \label{assumptConvergence}
	The coefficients $A,b$ of the numerical method (\ref{RewriteNumercialMethodForSchrodinger})  satisfy
	$$
	\left|a_{11}-1\right|+\left|a_{22}-1\right|+\left|a_{12}-\tau\right|+\left|a_{21}+\tau\right|=\mathcal{O}\left(\tau^2\right) \quad \text{and} \quad   \left|b_1-\alpha_1\right|+\left|b_2-\alpha_2\right|=\mathcal{O}(\tau).
	$$
\end{assumption}
It follows from 
\cite{chen2021asymptotically}  that Assumption \ref{assumptConvergence} ensures at least one order convergence of the numerical method in the mean-square sense. The main result on the 
asymptotic preservation of the LILs by stochastic symplectic methods is stated as follows.
\begin{thm} Let Assumptions \ref{AssumptionforGeneralFormula} and \ref{assumptConvergence} hold. For the stochastic symplectic  method, 
	the LILs for $\mathscr{X}(t)$ are  asymptotically preserved, i.e., \label{Asymptoticpreservation}
	\begin{align}
		&\lim_{M\rightarrow \infty} \lim _{\tau \rightarrow 0}\limsup_{n \rightarrow \infty} \frac{\left\|X^{M,\tau}(t_n)\right\|_{\mathbb{R}}}{\sqrt{ t_n \log \log t_n}}
		= \sqrt{\alpha_1^2+\alpha_2^2} \sup_k \sqrt{ \eta_k} \quad a.s., \label{XM} \\
		&\lim_{M\rightarrow \infty} \lim _{\tau \rightarrow 0}\limsup_{n \rightarrow \infty} \frac{\left\|Y^{M,\tau}(t_n)\right\|_{\mathbb{R}}}{\sqrt{ t_n \log \log t_n}}
		= \sqrt{\alpha_1^2+\alpha_2^2} \sup_k \sqrt{ \eta_k} \quad a.s.,\label{YM} \\
		&\lim_{M\rightarrow \infty} \lim _{\tau \rightarrow 0}\limsup_{n \rightarrow \infty} \frac{(\left\|X^{M,\tau}(t_n)\right\|^2_{\mathbb{R}}+\left\|Y^{M,\tau}(t_n)\right\|^2_{\mathbb{R}})^{\frac{1}{2}}}{\sqrt{ t_n \log \log t_n}}
		= \sqrt{\alpha_1^2+\alpha_2^2} \sup_k \sqrt{ \eta_k} \quad a.s. \label{XYM}
	\end{align} 
\end{thm}
\begin{proof}
	Combining Assumption \ref{assumptConvergence} and \cite[Lemma 4.2]{chen2021asymptotically}, we have
	$$
	\begin{aligned}
		&|a_{11}-1|+|a_{22}-1|+|a_{12}-\tau|+|a_{21}+\tau|=\mathcal{O}(\tau^2), \quad & |b_1-\alpha_1|+|b_2-\alpha_2|=\mathcal{O}(\tau), \\
		&\operatorname{tr}(A) \rightarrow 2 \quad \text{as} \quad \tau \rightarrow 0, \quad (1-\operatorname{tr}(A)+\det(A)) \sim \tau^2,
	\end{aligned}
	$$
	which leads to
	$$
	\begin{aligned}
		&\lim_{\tau\rightarrow 0}\frac{b_1-a_{11} b_1-a_{12} b_2}{\tau}=-\alpha_2, \quad \lim_{\tau\rightarrow 0} \frac{b_{1}\left(a_{11} b_{1}+a_{12} b_{2}\right)(2-\operatorname{tr}(A))}{\tau^2} =\alpha^2_1, \\
		&\lim_{\tau\rightarrow 0}\frac{b_2+a_{21} b_1-a_{11} b_2}{\tau} =-\alpha_1,  \quad \lim_{\tau\rightarrow 0} \frac{b_{2}\left(a_{21} b_{1}-a_{11} b_{2}\right)(2-\operatorname{tr}(A))}{\tau^2} =-\alpha^2_2.
	\end{aligned}
	$$
	Therefore, we obtain
	\begin{equation} \label{approxi1}
		\begin{aligned}
			\lim _{\tau \rightarrow 0} \xi_{k,1,\tau}&= \lim _{\tau \rightarrow 0} \frac{\left(b_1-\left(a_{11}b_1+a_{12}b_2\right)\right)^{2}+2 b_{1}\left(a_{11} b_{1}+a_{12} b_{2}\right)(1-\cos \theta_k)}{\left(1-\frac{\operatorname{tr}^{2}(A)}{4}\right)} \\
			&=  \lim _{\tau \rightarrow 0} \frac{4}{2+\operatorname{tr}(A)} \Big( \frac{\left(b_1-a_{11}b_1-a_{12}b_2\right)^{2}}{2-\operatorname{tr} (A)} + \frac{b_{1}\left(a_{11} b_{1}+a_{12} b_{2}\right)(2-\operatorname{tr}(A))}{2-\operatorname{tr} (A)} \Big)
			= \alpha_1^2+\alpha_2^2,
		\end{aligned}
	\end{equation}
	\begin{equation} \label{approxi2}
		\begin{aligned}
			\lim _{\tau \rightarrow 0} \xi_{k,2,\tau}&= \lim _{\tau \rightarrow 0} \frac{\left(b_2+\left(a_{21}b_1-a_{11}b_2\right)\right)^{2}+2 b_{2}\left(a_{21} b_{1}-a_{11} b_{2}\right)(\cos \theta_k -1)}{\left(1-\frac{\operatorname{tr}^{2}(A)}{4}\right)}  \\
			&=  \lim _{\tau \rightarrow 0} \frac{4}{2+\operatorname{tr}(A)} \Big( \frac{\left(b_2+a_{21}b_1-a_{11}b_2\right)^{2}}{2-\operatorname{tr} (A)} - \frac{b_{2}\left(a_{21} b_{1}-a_{11} b_{2}\right)(2-\operatorname{tr}(A))}{2-\operatorname{tr} (A)} \Big) 
			= \alpha_1^2+\alpha_2^2,
		\end{aligned}
	\end{equation} 
	which yields
	\begin{equation} \nonumber
		\lim _{\tau \rightarrow 0}\limsup_{n \rightarrow \infty} \frac{\left\|X^{M,\tau}(t_n)\right\|_{\mathbb{R}}}{\sqrt{ t_n \log \log t_n}}
		= \sqrt{\alpha_1^2+\alpha_2^2} \max_{k\in\{1,2,\ldots,M\}} \sqrt{ \eta_k} \quad a.s., 
	\end{equation}
	\begin{equation}\nonumber 
		\lim _{\tau \rightarrow 0}\limsup_{n \rightarrow \infty} \frac{\left\|Y^{M,\tau}(t_n)\right\|_{\mathbb{R}}}{\sqrt{ t_n \log \log t_n}}
		= \sqrt{\alpha_1^2+\alpha_2^2} \max_{k\in\{1,2,\ldots,M\}} \sqrt{ \eta_k} \quad a.s.
	\end{equation}
	Letting $M\rightarrow\infty$ and using the fact that
	$
	\lim_{M\rightarrow \infty} \max_{k\in\{1,2,\ldots,M\}} \sqrt{\eta_k} = \sup_k  \sqrt{\eta_k}\;  a.s., 
	$
	we derive \eqref{XM} and \eqref{YM}. 
	
	For the proof of \eqref{XYM}, by calculating 
	\begin{align*}
		c_{k,\tau,1}&:= -\cos(2\theta_{k})+2a_{11}\cos(\theta_k) -a_{11}^2+a_{12}a_{21} 
		=-\frac{\operatorname{tr}^2(A)}{2}+1+a_{11}\operatorname{tr}(A)-a_{11}^2+a_{12}a_{21} \\
		&\,=\frac{\operatorname{tr}(A)}{2}(-\operatorname{tr}(A)+2a_{11})+(1+a_{11})(1-a_{11})+a_{12}a_{21} \\
		&\,=\big((2-\operatorname{tr}(A))+(1+a_{11}-2)(1-a_{11})+a_{12}a_{21}\big)(1+\mathcal{O}(\tau)) \\
		&\,= \mathcal{O}(\tau^2)+\mathcal{O}(\tau^4)-\mathcal{O}(\tau^2)=\mathcal{O}(\tau^4), \\
		c_{k,\tau,2} &:=-a_{21}\cos(\theta_k)+a_{11}a_{21}=\mathcal{O}(\tau^3), \quad
		c_{k,\tau,3} :=a_{12}\cos(\theta_k)-a_{11}a_{12}=\mathcal{O}(\tau^3),
	\end{align*}
	we arrive at
	\begin{equation} \label{approxi3}
		\begin{aligned}
			\lim _{\tau \rightarrow 0} \xi_{k,3,\tau}&=\lim _{\tau \rightarrow 0} \Big(\frac{-b_1b_2 \cos(2\theta_k) +(b_2(a_{11} b_{1}+a_{12} b_{2})-b_1\left(a_{21} b_{1}-a_{11} b_{2}\right)) \cos (\theta_k) }{\frac{1}{4}(2+\operatorname{tr}(A))(2-\operatorname{tr}(A))} \\
			&\quad\quad \quad+  \frac{(a_{11} b_{1}+a_{12} b_{2})\left(a_{21} b_{1}-a_{11} b_{2}\right)  }{\frac{1}{4}(2+\operatorname{tr}(A))(2-\operatorname{tr}(A))} \Big) \\
			&=\lim _{\tau \rightarrow 0}\frac{\alpha_1\alpha_2c_{k,\tau,1}+\alpha_1^2c_{k,\tau,2}+\alpha_2^2c_{k,\tau,3}}{2-\operatorname{tr}(A)}
			=0.
		\end{aligned}
	\end{equation} 
	Combining \eqref{approxi1}, \eqref{approxi2}, and \eqref{approxi3}, for any given $\epsilon>0$, there exists $\tau_0(\epsilon)>0$ such that for all $\tau\in (0,\tau_0(\epsilon))$, it holds that
	\begin{align*}
		\alpha_1^2+\alpha_2^2-\epsilon< \xi_{k,1,\tau}<\alpha_1^2+\alpha_2^2+\epsilon, \quad \alpha_1^2+\alpha_2^2-\epsilon< \xi_{k,2,\tau}<\alpha_1^2+\alpha_2^2+\epsilon,\quad  -\epsilon<\xi_{k,3,\tau}<\epsilon.
	\end{align*}
	Hence, for all $\tau\in (0,\tau_0(\epsilon))$, 
	\begin{align*}
		&\quad \sup_{\sum^M_{k=1}(\rho^2_{1,k}+\rho^2_{2,k})=1}  \phi^{M,\tau}(\rho)\\
		&\leq \sup_{\sum^M_{k=1}(\rho^2_{1,k}+\rho^2_{2,k})=1} \sum^{M}_{k=1} \big((\rho^2_{1,k}+\rho^2_{2,k})(\alpha_1^2+\alpha_2^2+\epsilon)+\epsilon|\rho_{1,k}\rho_{2,k}|\big)\eta_k \\
		&\leq \big(\max_{k\in\{1,2,\ldots,M\}}\eta_k\big)   \sup_{\sum^M_{k=1}(\rho^2_{1,k}+\rho^2_{2,k})=1}\big( (\alpha_1^2+\alpha_2^2+\epsilon)  \sum^{M}_{k=1} (\rho^2_{1,k}+\rho^2_{2,k})+ \sum^{M}_{k=1}\frac{\epsilon}{2}(\rho^2_{1,k}+\rho^2_{2,k}) \big)\\
		&\leq (\alpha_1^2+\alpha_2^2+\frac{3}{2}\epsilon) \max_{k\in\{1,2,\ldots,M\}} \eta_k.
	\end{align*}
	There exists $k_0 \in \{1,2,\ldots,M\}$ such that $\eta_{k_0}= \max_{k\in\{1,2,\ldots,M\}}\eta_k$.  Let $$
	\rho_{1,k}= \Big\{
	\begin{aligned}
		&1, \quad k=k_0, \\
		&0, \quad k\neq k_0,
	\end{aligned}
	$$
	and $\rho_{2,k}=0$ for all $k\in\mathbb{N}^+$. Then it follows  that for all $\tau\in (0,\tau_0(\epsilon))$,
	\begin{align*}
		\sup_{\sum^M_{k=1}(\rho^2_{1,k}+\rho^2_{2,k})=1}  \phi^{M,\tau}(\rho)
		&\geq\xi_{k_0,1,\tau}\eta_{k_0}  \geq (\alpha_1^2+\alpha_2^2-\epsilon)\max_{k\in\{1,2,\ldots,M\}}\eta_k.
	\end{align*}
	This leads to that for all $\tau\in (0,\tau_0(\epsilon))$,
	\begin{equation} \nonumber
		(\alpha_1^2+\alpha_2^2-\epsilon)\max_{k\in\{1,2,\ldots,M\}}\eta_k \leq \sup_{\sum^M_{k=1}(\rho^2_{1,k}+\rho^2_{2,k})=1}  \phi^{M,\tau}(\rho) \leq (\alpha_1^2+\alpha_2^2+\frac{3}{2}\epsilon) \max_{k\in\{1,2,\ldots,M\}} \eta_k, 
	\end{equation} 
	which is equivalent to
	$
	\lim_{\tau\rightarrow 0}\sup_{\sum^M_{k=1}(\rho^2_{1,k}+\rho^2_{2,k})=1}  \phi^{M,\tau}(\rho)
	=(\alpha_1^2+\alpha_2^2) \max_{k\in\{1,2,\ldots,M\}}\eta_k.
	$
	Hence, we have
	\begin{equation}\nonumber 
		\lim _{\tau \rightarrow 0}\limsup_{n \rightarrow \infty} \frac{\left\|u^{M,\tau}(t_n)\right\|_{\mathbb{R}}}{\sqrt{ t_n \log \log t_n}}= \sqrt{\alpha_1^2+\alpha_2^2} \max_{k\in\{1,2,\ldots,M\}} \sqrt{\eta_k} \quad  a.s. 
	\end{equation}
	Letting $M\rightarrow \infty$ completes the proof of \eqref{XYM}.
\end{proof}

If the numerical method \eqref{RewriteNumercialMethodForSchrodinger} applied to the SHS does not preserve the symplecticity, then the method is called non-symplectic method.  
Notice that for the case $\det(A)>1$, the upper limit of the solution scaled by polynomials of $t$ may not exist in $\mathbb{R}$. We only discuss the case $0<\det(A)<1$ here, for which we show that the upper limit of the solution scaled by $t^{\epsilon}$ with any given $\epsilon>0$ equals 0 almost surely. This is stated as follows.
\begin{thm} \label{LILfornon-symplecticSch}
	Let Assumptions \ref{AssumptionforGeneralFormula} and \ref{assumptConvergence} hold. For the non-symplectic method with $0<\det (A)<1$, 
	$\mathscr{X}^{M,\tau}_{\mathrm{n\text{-}sym}}(t_n) \in \{\|X^{M,\tau}(t_n)\|_{\mathbb{R}}, \|Y^{M,\tau}(t_n)\|_{\mathbb{R}}, (\|X^{M,\tau}(t_n)\|^2_{\mathbb{R}}+\|Y^{M,\tau}(t_n)\|^2_{\mathbb{R}})^{\frac{1}{2}} \}$ satisfies that 
	\begin{align*}
		&\limsup _{n \rightarrow \infty}\frac{\mathscr{X}^{M,\tau}_{\mathrm{n\text{-}sym}}(t_n)}{t_n^\epsilon}=0\quad a.s., \quad \forall \, \epsilon>0.
	\end{align*}
\end{thm}
\begin{proof}
	We only give the proof for $\|X^{M,\tau}(t_n)\|_{\mathbb{R}}$ here since the proofs for $\|Y^{M,\tau}(t_n)\|_{\mathbb{R}}$ and $(\|X^{M,\tau}(t_n)\|^2_{\mathbb{R}}+\|Y^{M,\tau}(t_n)\|^2_{\mathbb{R}})^{\frac{1}{2}}$  are similar.  Define $f(t_n) := t_n^{\epsilon/2}$ with $\epsilon>0$.
	Based on $0<\det(A)<1$, we have 
	\begin{equation} \nonumber
		\begin{aligned}
			\sum_{j=0}^{n-2} (\hat{\alpha}^k_{j})^2 \leq K_{1}(\theta_k), \quad
			\big|2 \sum_{j=1}^{n-1} \hat{\alpha}^k_{j} \hat{\alpha}^k_{j-1}\big| \leq K_{2}(\theta_k),
		\end{aligned}
	\end{equation}
	which imply that
	\begin{equation} \label{non-symplecticVarM_n}
		\begin{aligned}
			&\quad \operatorname{Var}(G^{M,\tau,1}_{\rho}(n))  \\
			&=\sum_{k=0}^{M}\sum_{j=0}^{n-1} \rho^2_{1,k}\big(-b_{1} \hat{\alpha}^k_{n-2-j}+\left(a_{11} b_{1}+a_{12} b_{2}\right) \hat{\alpha}^k_{n-1-j}\big)^2 \eta_k \tau  \\
			&\leq \sum_{k=0}^{M} 
			\eta_k \tau \big([b_{1}^{2}+\left(a_{11} b_{1}+a_{12} b_{2}\right)^{2}] K_1(\theta_k)+ |b_{1}\left(a_{11} b_{1}+a_{12} b_{2}\right) |K_2(\theta_k) +\left(a_{11} b_{1}+a_{12} b_{2}\right)^{2} (\hat{\alpha}^k_{n-1})^2\big)\\
			&=: K(\theta) <\infty.
		\end{aligned}
	\end{equation}
	Since $G^{M,\tau,1}_{\rho}(n)$ is a Gaussian random variable, we obtain
	\begin{equation} \nonumber 
		\begin{aligned}
			\mathbb{P}\left\{G^{M,\tau,1}_{\rho}(n) > f(t_n)\right\} &\leq \frac{1}{\sqrt{2\pi}} \frac{1}{\frac{f(t_n)}{\sqrt{\operatorname{Var}(G^{M,\tau,1}_{\rho}(n))}}} \exp\big\{ - \frac{t_n^{\epsilon}}{2\operatorname{Var}(G^{M,\tau,1}_{\rho}(n))}\big\} \\
			&\leq \frac{1}{\sqrt{2\pi}} \frac{\sqrt{K(\theta)}}{(n\tau)^{\epsilon/2}} \exp\big\{ - \frac{(n\tau)^{\epsilon}}{2K(\theta)}\big\}\leq \sqrt{\frac{K(\theta)}{2\pi \tau^{\frac{\epsilon}{2}}}} \exp\big\{ - \frac{(n\tau)^{\epsilon}}{2K(\theta)}\big\} ,
		\end{aligned}
	\end{equation}
	which leads to
	$
	\sum^{\infty}_{n=1} \mathbb{P}\left\{G^{M,\tau,1}_{\rho}(n) > f(t_n)\right\} < \infty.
	$
	By the Borel--Cantelli lemma, we arrive at
	$
	\mathbb{P}\Big\{ \big\{G^{M,\tau,1}_{\rho}(n) > f(t_n)\big\} \quad i.o.\Big\}=0,
	$
	which yields
	$
	\limsup\limits_{n \rightarrow \infty} \frac{G^{M,\tau,1}_{\rho}(n)}{t_n^{\epsilon/2}} \leq 1 \; a.s.
	$
	Therefore, it shows that		
	$
	\limsup\limits_{n \rightarrow \infty} \frac{G^{M,\tau,1}_{\rho}(n)}{t_n^{\epsilon}} = \limsup\limits_{n \rightarrow \infty} \frac{G^{M,\tau,1}_{\rho}(n)}{t_n^{\epsilon/2}} \frac{t_n^{\epsilon/2}}{t_n^{\epsilon}} 
	\leq \lim\limits_{n \rightarrow \infty} \frac{t_n^{\epsilon/2}}{t_n^{\epsilon}}
	=0 \; a.s.
	$
	Since $\operatorname{Var}(-G^{M,\tau,1}_{\rho}(n))$ has the same expression as \eqref{non-symplecticVarM_n}, we have 
	$
	\liminf\limits_{n \rightarrow \infty} \frac{G^{M,\tau,1}_{\rho}(n)}{t_n^{\epsilon}} = -\limsup\limits_{n \rightarrow \infty} \frac{-G^{M,\tau,1}_{\rho}(n)}{t_n^{\epsilon/2}} \frac{t_n^{\epsilon/2}}{t_n^{\epsilon}} 
	\geq 0 \; a.s.
	$
	Hence, $	\lim\limits_{n \rightarrow \infty} \frac{G^{M,\tau,1}_{\rho}(n)}{t_n^{\epsilon}}=0 \; a.s.,$
	which finishes the proof.
\end{proof}
\section{Applications to finite-dimensional and infinite-dimensional SHSs} \label{SectionApplication}
In this section, we give the applications of theoretical results to concrete examples, including the linear stochastic oscillator and the linear stochastic Schr\"{o}dinger equation,  respectively. 
\subsection{LIL for linear stochastic oscillator}
Consider the linear stochastic oscillator $\ddot{X}(t)+X(t)=\alpha \dot{W}(t)$ with $\alpha>0$ and $W(t)$ being a one-dimensional standard Brownian motion defined on $(\Omega,\mathcal{F},\{\mathcal{F}_t\}_{t\geq 0},\mathbb{P})$. Let $\mathbb{U}=\mathbb{R}, B(x)=x,  \alpha_1=0,\alpha_2=\alpha$, the linear stochastic oscillator can be rewritten into an SHS in the form of \eqref{GeneralSHS}, namely,
\begin{equation}
	\nonumber
	d\begin{pmatrix}
		X(t) \\
		Y(t)
	\end{pmatrix}=\begin{pmatrix}
		0 & 1 \\
		-1 & 0
	\end{pmatrix}\begin{pmatrix}
		X(t) \\
		Y(t)
	\end{pmatrix} d t+\alpha\begin{pmatrix}
		0 \\
		1
	\end{pmatrix}d W(t),\quad \begin{pmatrix}
		X(0) \\
		Y(0)
	\end{pmatrix}=\begin{pmatrix}
		X_{0} \\
		Y_{0}
	\end{pmatrix}.
\end{equation} 
	By Theorems \ref{LILContinuous} and \ref{LILfordiscrete}, 
	we have that
		the exact solution $X(t)$ of the linear stochastic oscillator and the numerical solution $X_n$ of the stochastic symplectic method  obey the  LILs:
		\begin{equation} \label{LILforOscill}
			\limsup_{t \rightarrow \infty} \frac{|X(t)|}{\sqrt{ t \log \log t}}= \alpha \quad a.s., \quad \quad
			\limsup_{n \rightarrow \infty} \frac{|X_n|}{\sqrt{ t_n \log \log t_n}}= \sqrt{\xi_{\tau}} \quad a.s.,
		\end{equation}
		respectively, where $\xi_{\tau} := \frac{b_{1}^{2}+\left(a_{11} b_{1}+a_{12} b_{2}\right)^{2}-2 b_{1}\left(a_{11} b_{1}+a_{12} b_{2}\right) \cos \left(\theta\right)}{ \sin ^{2}\left(\theta\right)}$.
		Further, by Theorem \ref{Asymptoticpreservation}, we derive that the stochastic symplectic method asymptotically preserves the LIL of the exact solution:
		$$
		\lim_{\tau \rightarrow 0}\limsup_{n \rightarrow \infty} \frac{|X_n|}{\sqrt{ t_n \log \log t_n}}=\limsup_{t \rightarrow \infty} \frac{|X(t)|}{\sqrt{ t \log \log t}}= \alpha \quad a.s.
		$$
	\subsection{LIL for linear stochastic Schr\"{o}dinger equation}
	Consider the linear stochastic Schr\"{o}dinger equation 
	$d u(t)  =\mathbf{i} \Delta u(t) d t+\mathbf{i} \alpha d W(t), \; t>0$ with 
	$u(0)  =u_0, 
	$
	where $\alpha>0$, $\Delta$ denotes the Laplace operator with Dirichlet boundary conditions, and $W(t)$ is an $L^2(0,\pi;\mathbb{R})$-valued $Q$-Wiener process  on a complete filtered probability space $(\Omega,\mathcal{F},\{\mathcal{F}_t\}_{t\geq 0},\mathbb{P})$.
	
	Let $\mathbb{U}=L^2(0, \pi ; \mathbb{R}), \mathbb{H}=L^2(0, \pi ; \mathbb{C})$. The corresponding complex inner product and real inner product are defined as  $\langle f, g \rangle_{\mathbb{C}} = \int^\pi_0 f(\zeta)\bar{g}(\zeta) d \zeta$ and $\langle f, g \rangle_{\mathbb{R}} =  \Re \int^\pi_0 f(\zeta)\bar{g}(\zeta) d \zeta$ for $f,g \in \mathbb{H}$, respectively. The sequence $\big\{e_k:e_k(\zeta)=\sqrt{\frac{2}{\pi}} \sin (k \zeta), \zeta\in[0,\pi]\big\}_{k\in \mathbb{N}^+}$ forms an orthonormal basis of both $\mathbb{U}$ and $\mathbb{H}$. Set $B=-\Delta, \alpha_1=0, \alpha_2=\alpha, X_0=\Re u_0, Y_0=\Im u_0$.
	Let $u(t) =X(t)+\mathbf{i}Y(t)$, then $X(t), Y(t)$ can be written into
	an SHS in the form of \eqref{GeneralSHS}, namely, 
	\begin{equation} \nonumber
		d\begin{pmatrix}
			X(t) \\
			Y(t)
		\end{pmatrix}=\begin{pmatrix}
			0 & -\Delta \\
			\Delta & 0
		\end{pmatrix} \begin{pmatrix}
			X(t) \\
			Y(t)
		\end{pmatrix} d t+\alpha\begin{pmatrix}
			0 \\
			1
		\end{pmatrix}d W(t),\quad \begin{pmatrix}
			X(0) \\
			Y(0)
		\end{pmatrix}=\begin{pmatrix}
			\Re u_{0} \\
			\Im u_{0}
		\end{pmatrix}.
	\end{equation} 
	By Theorem \ref{LILContinuous}, 
	we have the LIL of 
		the exact solution $u(t)$:
		\begin{equation} \nonumber
			\limsup_{t \rightarrow \infty} \frac{\left\|u\left(t\right)\right\|_{\mathbb{R}}}{\sqrt{ t \log \log t}}  = 	\limsup_{t \rightarrow \infty} \frac{(\|X(t)\|^2_{\mathbb{R}}+\|Y(t)\|^2_{\mathbb{R}})^{\frac{1}{2}}}{\sqrt{ t \log \log t}} =\alpha \sup_{j\in \mathbb{N}^+} \sqrt{\eta_j} \quad a.s.
		\end{equation}
		According to Theorem \ref{LILfordiscrete}, the solution $u^{M,\tau}_n=X^{M,\tau}_n+\mathbf{i}Y^{M,\tau}_n$ of the stochastic symplectic method obeys the following LIL:
		$$\limsup_{n \rightarrow \infty} \frac{\|u^{M,\tau}_n\|_{\mathbb{R}}}{\sqrt{ t_n \log \log t_n}}= \sup_{\sum^M_{k=1}(\rho^2_{1,k}+\rho^2_{2,k})=1}  \sqrt{\phi^{M,\tau}(\rho)} \quad a.s.,$$
		where  $\phi^{M,\tau}(\rho)$ is given in Theorem \ref{LILfordiscrete}.
		Further, by Theorem \ref{Asymptoticpreservation}, we derive that the stochastic symplectic method asymptotically preserves the LIL of the exact solution:
		$$
		\lim_{M \rightarrow \infty}\lim_{\tau \rightarrow 0}\limsup_{n \rightarrow \infty} \frac{\|u^{M,\tau}_n\|_{\mathbb{R}}}{\sqrt{ t_n \log \log t_n}}=\limsup_{t \rightarrow \infty} \frac{\|u(t)\|_{\mathbb{R}}}{\sqrt{ t \log \log t}}= \alpha \sup_{j\in \mathbb{N}^+} \sqrt{\eta_j} \quad a.s.
		$$

	\bibliographystyle{plain}
	\bibliography{references}

\end{document}